\documentclass[12pt]{amsart}
\usepackage{amsfonts,amsmath,amssymb,amscd,amsthm}
\usepackage{amsaddr} 
\usepackage{faktor}    
\usepackage[utf8]{inputenc}
\usepackage[english]{babel}
\usepackage{pst-node}
\usepackage{enumitem}
\usepackage{tikz-cd}
\usepackage{pdfpages}
\usepackage{mathrsfs}
\usepackage[left=30mm, top=20mm, right=30mm, bottom=24mm]{geometry} 
\long\def\comment#1\endcomment{}
\usepackage[unicode, pdftex]{hyperref}
\usepackage{amsrefs}

\DeclareMathOperator{\spam}{span}

\DeclareMathOperator{\Ra}{Re}

\DeclareMathOperator{\rank}{rank}
\DeclareMathOperator{\Tr}{Tr}

\DeclareMathOperator{\Hom}{Hom}
\DeclareMathOperator{\Sp}{Sp}
\DeclareMathOperator{\SL}{SL}
\DeclareMathOperator{\PSL}{PSL}
\DeclareMathOperator{\GL}{GL}

\newtheorem{thm}{Theorem}[section]

\newtheorem{cor}[thm]{Corollary}
\newtheorem{lem}[thm]{Lemma}
\newtheorem{prop}[thm]{Proposition}
\newtheorem{defin}[thm]{Definition}

\theoremstyle{remark}

\numberwithin{equation}{section}

\title[Flexible Hilbert-Schmidt Stability versus Hyperlinearity]{Flexible Hilbert-Schmidt Stability versus Hyperlinearity for Property (T) Groups}
\author{Alon Dogon}
\address{Department of Mathematics, Weizmann Institute of Science, Rehovot, Israel}
\email{alon.dogon@mail.huji.ac.il}
\thanks{Funding: This work was supported by the European Research Council (ERC) 
under the European Unions Horizon 2020 research and innovation program (Grant No. 882751).}
\thanks{\textbf{keywords:} Kazhdan's property (T), hyperlinear groups, group stability, random groups, lattices in Lie groups.}

\begin{document}

\maketitle
\begin{abstract}
We prove a statement concerning hyperlinearity for central extensions of property (T) groups in the presence of flexible HS-stability, and more generally, weak ucp-stability.
Notably, this result is applied to show that if $\Sp_{2g} (\mathbb Z)$ is flexibly HS-stable, then there exists a non-hyperlinear group.
Further, the same phenomenon is shown to hold generically for random groups sampled in Gromov's density model,
as well as all infinitely presented property (T) groups. 
This gives new directions for the possible existence of a non-hyperlinear group. Our results yield Hilbert-Schmidt analogues for Bowen and Burton's work relating flexible P-stability of $\PSL_n(\mathbb Z)$ and the existence of non-sofic groups \cite{BB}.
\end{abstract}

\section{Introduction} \label{sec:intro}

Hyperlinearity is an approximation property for groups that was introduced by Radulescu \cite{Radulescu} following the seminal work of Connes \cite{Connes}.
Roughly speaking, a group $\Gamma$ is said to be \emph{hyperlinear} if for every finite window $F \subset \Gamma$, one can find unitary matrices that \emph{approximately} realise the partial multiplication given by $F$. 
Initial interest in hyperlinear groups stemmed from the fact that the group von Neumann algebra of a hyperlinear group is Connes embeddable (see \cite{Radulescu}), making them relevant for Connes' embedding problem. 
Even though Connes' embedding problem has recently been solved in the breakthrough work \textsc{MIP$^*$=RE}  \cite{MIP*=RE}, the notorious question of whether a non-hyperlinear group exists remains open.
The goal of the present paper is to provide new possible avenues in approaching this problem.

To define hyperlinear groups precisely, we first introduce some notation: 
the \emph{(normalised) Hilbert-Schmidt inner product} on $M_n(\mathbb C)$ is given by $\langle x, y \rangle_{\tau_n} = \tau_n (y^*x)$, where $\tau_n(x) = (1/n) \Tr(x)$ denotes the normalized trace on $M_n(\mathbb C)$ and $x, y \in M_n(\mathbb C)$. The induced norm $\Vert x \Vert_{2, \tau_n}$ is the \emph{(normalized) Hilbert-Schmidt norm}. Let $U(n) \subset M_n(\mathbb C)$
be the appropriate group of unitaries.

\begin{defin} \label{def:hyp_intro}
A group $G$ is said to be \textbf{hyperlinear}, or \textbf{Connes embeddable}, if for every $\epsilon>0$ and every finite subset $F \subset G$, there 
exists $d \in \mathbb N$ and a map $\varphi: F \to U(d)$ such that:
$$ \Vert \varphi(g)\varphi(h) - \varphi(gh) \Vert_{2,\tau_d} < \epsilon \; \text{ for all $g,h \in F$ such that $gh \in F$ }$$
and 
$$ \Vert \varphi(g) - 1 \Vert_{2,\tau_d} >  \sqrt{2} - \epsilon \; \text{ for all $g \in F \setminus \{1_\Gamma\}$}.$$

\end{defin}

Further interest in hyperlinear groups came from their relation to \emph{sofic groups}: Soficity is a finitary approximation property for groups that was introduced by Gromov \cite{Gromov} and further studied by Weiss \cite{Weiss},
and is known to hold for a large variety of groups (including all \emph{initially subamenable} groups, see \cite{Pestov}). 
Similarly to hyperlinearity, it is a major open problem to determine whether there exists a non-sofic group.
The work of Elek and Szab{\'o} \cite{ES2} showed sofic groups are hyperlinear, thus, an example of a non-hyperlinear group would also be an example of a non-sofic group. 
Both hyperlinear and sofic groups have found many applications in group theory, measurable and topological dynamics and operator algebras (see for example \cite{Bowen}, \cite{ES1}, \cite{ES2}, \cite{KT}, \cite{Thom2}, \cite{Jaikin}), making the above questions all the more interesting. We refer to \cite{Pestov}, \cite{CL} for excellent surveys on this topic.
   


Recently, Bowen and Burton \cite{BB} gave a surprising new approach towards finding non-sofic groups, by means of \emph{flexible P-stability}:
they showed that if $\PSL_n(\mathbb Z)$ is \emph{flexibly stable with respect to permutations}
for some $n \geq 5$, then there exists a non-sofic group. We refer to the original paper \cite{BB} for precise definitions and statement of this result 
(see also the discussion following Corollary \ref{cor:Sp}). 

We prove statements of similar nature in the setting of \emph{flexible 
Hilbert-Schmidt stability} (Definition \ref{def:flex_stab}) and hyperlinearity of general property (T) groups. 
Our techniques apply to a large family, including Gromov random groups, lattices in some Lie groups, and
all infinitely presented property (T) groups. As a consequence, \emph{if flexibly Hilbert-Schmidt stable
groups occur with positive probability in Gromov's model,
then there are non-hyperlinear groups} (see Corollary \ref{cor:rand}).
Let us recall the setting of stability in the Hilbert-Schmidt norm:

\begin{defin}[Becker and Lubotzky \cite{BL}, cf. Ioana \cite{Ioana2}]\label{def:flex_stab}
Given a countable group $\Gamma$, a sequence of maps $\varphi_n: \Gamma \to U(d_n)$ (for some $d_n \in \mathbb N$)
is called an \textbf{asymptotic homomorphism} if for all $g,h \in \Gamma$ one has:
$$\lim_n \Vert \varphi_n(g) \varphi_n(h) - \varphi_n(gh) \Vert_{2, \tau_{d_n}}=0 .$$ 
The group $\Gamma$ is said to be \textbf{flexibly HS-stable} if for every asymptotic homomorphism $\varphi_n: \Gamma \to U(d_n)$
there exists a sequence of true homomorphisms $\pi_n : \Gamma \to U(D_n)$ for some $D_n \geq d_n$ 
with $\lim_{n \to \infty} \frac{D_n}{d_n} = 1$ such that:
$$ \lim_n \Vert  \varphi_n(g) - P_n \pi_n(g) P_n\Vert_{2, \tau_{d_n}} = 0 \text{ for all $g \in \Gamma$,}$$
where $P_n: \mathbb C^{D_n} \to \mathbb C^{d_n}$ is the projection onto the first $d_n$-coordinates.
\end{defin}

That is, a group $\Gamma$ is flexibly HS-stable if every sequence of maps into unitary groups which \emph{asymptotically} 
resemble a homomorphism is close to a sequence of \emph{compressions} of \emph{genuine}
homomorphisms into unitary groups of slightly larger dimensions. 
Both flexible HS-stability, and a stricter notion called \emph{HS-stability}
\footnote{A group $\Gamma$ is HS-stable if it satisfies definition \ref{def:flex_stab} with $D_n = d_n$ for all $n$.}
have attracted considerable amount of attention in recent years (see \cite{Ioana2}, \cite{ISW},  \cite{delaSalle}, \cite{LV}, \cite{EckShul}, \cite{GS}, \cite{HS}, 
\cite{AP}).

Many amenable groups have been shown to be (flexibly) HS-stable, including all finitely generated 
virtually nilpotent groups, certain lamplighter groups and the Baumslag-Solitar groups $\text{BS}(1,n)$ \cite{LV}.
Non-amenable examples include certain one-relator groups (Theorem 8 in \cite{HS}) and virtually free groups \cite{GS}. 
One can obtain further examples by taking direct products with HS-stable amenable groups \cite{IS}.

On the other hand, any finitely generated hyperlinear group that is not residually finite cannot be flexibly HS-stable \cite{BL}.
Ioana, Spaas and Wiersma \cite{ISW} showed that $SL_2(\mathbb Z) \ltimes \mathbb Z^2$ is not flexibly HS-stable,
giving the first residually finite example. For groups of very different nature, Ioana \cite{Ioana2} proved that 
$\mathbb F_m \times \mathbb F_k$ is not flexibly HS-stable, $\mathbb F_m, \mathbb F_k$ being non-abelian free groups.

The following is our main result:

\begin{thm} \label{thm:main}
Let $\Gamma$ be a countable property (T) group and let $A$ be a countable torsion-free abelian group.
Assume there is a non-split central extension of the form 
\begin{center}
\begin{tikzcd}
1 \arrow[r] & A \arrow[r] & G \arrow[r] & \Gamma \arrow[r] & 1.
\end{tikzcd}
\end{center}

Further, assume either that $\Gamma$ is perfect, or that $G$ has property (T) and $A = \mathbb Z$. 
If $\Gamma$ is flexibly HS-stable, then $G$ is not hyperlinear.
\end{thm}

Flexible HS-stability was defined by Becker and Lubotzky \cite{BL} after they proved that infinite hyperlinear 
property (T) groups are never HS-stable, in hopes that some might be flexibly HS-stable. The result above, coupled
with Corollaries \ref{cor:Sp} and \ref{cor:rand}, show that for many property (T) groups,
establishing flexible HS-stability would in fact give the existence of a non-hyperlinear group. Thus,
deciding flexible HS-stability in this scenario is extremely interesting. With that said, there is no
known example of a group that is flexibly HS-stable, but not HS-stable.

In fact, we can show that a significantly weakened version of flexible HS-stability, called
\emph{weak ucp-stability}, is enough to establish the non-hyperlinearity guaranteed in Theorem
\ref{thm:main}. Recall the notion of a hyperlinear approximation (see \cite{Burton}):

\begin{defin}\label{def:hyp}
Given a countable group $\Gamma$, a sequence of maps $\varphi_n : \Gamma \to U(d_n)$ is said to be a 
\textbf{hyperlinear approximation} if it is an asymptotic homomorphism, 
and for every nontrivial $g \in \Gamma$, we have $\liminf_n \Vert \varphi(g) - 1 \Vert_{2, \tau_{d_n}} \geq \sqrt{2}   $.

\end{defin}

It follows from the definitions that $\Gamma$ is hyperlinear if an only if it admits a hyperlinear approximation.

\begin{defin}\label{def:weak_ucp_stab}
A countable group $\Gamma$ is said to be \textbf{weakly ucp-stable} if for every sequence of finite dimensional Hilbert
spaces $\mathcal H_n$ and a hyperlinear approximation $\varphi_n: \Gamma \to U(\mathcal H_n)$
there exists a sequence of homomorphisms $\pi_n : \Gamma \to U(\widehat{\mathcal H}_n)$ for some 
(possibly infinite dimensional) Hilbert spaces $\widehat{\mathcal H}_n$ containing $\mathcal H_n$ with
$$ \Vert  \varphi_n(g) - P_n \pi_n(g) P_n\Vert_{2, \tau_{\dim \mathcal H_n}} \xrightarrow[n \to \infty]{} 0
\text{ for all $g \in \Gamma$,}$$
where $P_n: \widehat{\mathcal H}_n \to \mathcal H_n$ is the projection onto $\mathcal H_n$.
\end{defin}

Weak ucp-stability relaxes flexible HS-stability in two regards: we only require a rigidity property for 
hyperlinear approximations, and we only require the existence of possibly \emph{infinite} 
dimensional unitary representations whose compressions comprise of corrections. This notion is a combination of 
definitions from \cite{BL}, \cite{AK}. The term "ucp" (unital completely positive) stems from the fact that maps
of the form $g \mapsto P_n \pi_n(g) P_n$ are restrictions of ucp maps from the maximal group C$^*$-algebra
$C^*(\Gamma)$ to $\Gamma$ (see Section \ref{prelim:ucp}).

The following is a strengthened version of our main theorem: 

\begin{thm} \label{thm:main2}
Under the assumptions of Theorem \ref{thm:main}, if $\Gamma$ is weakly ucp-stable, then $G$ is not hyperlinear.
\end{thm}

Since flexible HS-stability implies weak ucp-stability, Theorem \ref{thm:main2} obviously implies \ref{thm:main}.
Weak ucp-stability is a strictly weaker property when compared with flexible HS-stability:
indeed, by Corollary 1.7 in \cite{ISW} any group $\Gamma$ such that $C^*(\Gamma)$ possesses
Kirchberg's local lifting property (LLP) is weakly ucp-stable.
Thus, by the Choi-Effors lifting theorem
\cite{CE}, any amenable group is weakly ucp-stable. Yet, an amenable group which is not
residually finite cannot be flexibly HS-stable (see \cite{BL}).

The proof of Theorem \ref{thm:main2} is inspired by techniques of Ioana, Spaas and Wiersma \cite{ISW}, and it goes as follows: Assume by contradiction that $G$ is hyperlinear and $\Gamma$ is weakly ucp-stable.
By a lemma of Thom \cite{Thom1}, hyperlinearity of $G$ implies that every character $\chi$ of the central subgroup $A$ defines a projective unitary
representation $\pi_\chi$ of $\Gamma$ that generates a Connes embeddable von Neumann algebra.
It can then be shown that there exist $\chi_n \in \widehat{A}$ arbitrarily close to the trivial character, whose
projective representations $\pi_{\chi_n}$ are not cohomologous to true unitary representations.
Using Connes embeddability, we can find finite dimensional approximations for $\pi_{\chi_n}$. As $\chi_n$ was almost
trivial, these finite dimensional approximations give a hyperlinear approximation of $\Gamma$.
By the stability assumption, we obtain a sequence of ucp maps $\varphi_n$ which approximate $\pi_{\chi_n}$. 
From here, we proceed as in the proof of Theorem 4.3 in \cite{ISW} and obtain a contradiction to property (T).

\subsection*{\textbf{Gromov random groups.}} 
Somewhat surprisingly, the phenomenon described in Theorem \ref{thm:main2} can be shown to be present with overwheliming probability in a random finitely presented group:
Let $\mathbb F$ denote the free group on $k$ generators, freely generated by $X = \{x_1, \dots x_k\}$, for some $k \geq 2$. Denote by 
$\mathbb S_l$ the sphere of radius $l$ in the Cayley graph of $\mathbb F$ with respect to $X \cup X^{-1}$, that is, the 
set of reduced words in $x_i^{\pm 1}$ of length $l$ in $\mathbb F$. 
Fix a \emph{density} $0\leq d$ and pick 
$\lfloor \vert \mathbb S_l \vert^d \rfloor = \lfloor ((2k)(2k-1)^{l-1})^d \rfloor$ independently uniformly sampled 
elements from $\mathbb S_l$ (with repetitions). Denote the resulting set of relations by $R$, and consider the finitely presented group 
$\Gamma =  \langle X \vert R \rangle$. In this case, $\Gamma$ is called a \emph{Gromov random group at density $d$ with $k$ generators
and relators of length $l$}, denoted by $\Gamma \sim \mathcal G(k,l,d)$. 
For technical reasons, we will only consider lengths $l$ divisible by $3$.
A property $P$ of groups is said to hold \emph{with overwhelming probability (w.o.p.)} if with probability going to $1$, as $l \to \infty$,
a group $\Gamma \sim \mathcal G(k,3l,d) $ will satisfy $P$, that is: 
$$\lim_{l \to \infty} \mathbb P\left(\Gamma \sim \mathcal G(k,3l,d)\text{ satisfies }P\right) = 1.$$

\begin{prop} \label{prop:rand}
Given $ 1/3 < d < 1/2$, with overwhelming probability,  $\Gamma \sim \mathcal G(k,3l,d)$ will satisfy the 
assumptions of Theorem \ref{thm:main2}.
\end{prop}
As an immediate consequence, we get the following:

\begin{cor} \label{cor:rand}
Given $1/3 < d < 1/2$, if
$$ \limsup_{l \to \infty} \mathbb P \left(\Gamma \sim \mathcal G(k,3l,d)\;\text{ is weakly ucp-stable}\right) > 0,$$
then there exists a non-hyperlinear group.
\end{cor}

One may also view Corollary \ref{cor:rand} as a statement about random sets of equations over groups, once stability is interpreted
using generators and relations (see Lemma \ref{lem:ucp_stab_defs}).

It is a well known theorem of Gromov (Theorem 11 in \cite{Ollivier}) that at such densities, $\Gamma$ will be a torsion-free hyperbolic group with overwhelming probability. 
To the best of the author's knowledge, the only examples of hyperbolic groups which are known to be flexibly HS-stable are virtually free groups (see \cite{GS}).
It also does not seem to be known whether there exists a hyperbolic group that is not flexibly HS-stable.
Thus, a natural direction of further research is to find more examples/non-examples of flexible HS-stability in the hyperbolic realm.
For more information on Gromov random groups, see \cite{Ollivier}. 

\subsection*{\textbf{Lattices in non-simply connected Lie groups.}}
Another interesting consequence of Theorem \ref{thm:main2} is the following:

\begin{cor} \label{cor:Sp}
Let $G$ be a connected semisimple Lie group with infinite cyclic fundamental group and with property (T), and let
$\Gamma \leq G$ be a lattice. Let $\widetilde{G} \twoheadrightarrow G$ be the universal cover, 
$\widetilde{\Gamma} \leq \widetilde{G}$ be the pullback of $\Gamma$ under the covering map. If $\Gamma$ is weakly ucp-stable,
then $\widetilde{\Gamma}$ is not hyperlinear.
\end{cor}

A concrete case to keep in mind is of the symplectic group $G = \Sp_{2g}(\mathbb R)$, $g \geq 2$ and its lattice of
integer points $\Sp_{2g}(\mathbb Z)$. As $\pi_1(\Sp_{2g}(\mathbb R)) = \mathbb Z$, we get a a central extension of the form:

\begin{center}
\begin{tikzcd}
1 \arrow[r] & \mathbb Z \arrow[r] & \widetilde{\Gamma} \arrow[r] & \Sp_{2g}(\mathbb Z) \arrow[r] & 1,
\end{tikzcd}
\end{center}

where $\widetilde{\Gamma}$ is a lattice in the universal cover $\widetilde{\Sp}_{2g}(\mathbb R)$. The group $\widetilde{\Gamma}$ has been shown to be non-residually finite by Deligne \cite{Deligne},
and the corollary gives a potential way to show that it is  not even hyperlinear (or sofic).

Intrestingly, the work of Bader, Lubotzky, Sauer and Weinberger \cite{BLSW} shows that $\widetilde{\Sp}_{2g}(\mathbb Z)$ is \emph{not} Frobenius-approximated for $g\geq 3$.
This means that it does not satisfy the analog of Definition \ref{def:hyp_intro} where instead of using the
normalised Hilbert-Schmidt norm, we consider it without normalisation (the \emph{Frobenius norm}). 
Therefore, deciding whether these groups are hyperlinear becomes even more interesting.
With that said, there is no known relation between hyperlinearity and Frobenius-approximability.
We refer to \cite{DGLT} for more on Frobenius-approximation and stability.

Let us comment on how our work compares with the result of Bowen and Burton \cite{BB}, stating that if
$\PSL_d(\mathbb Z)$ were flexibly P-stable for some $d \geq 5$, then there exists a non-sofic group. 
Flexible P-stability can be defined analogously to flexible HS-stability, simply by replacing unitary groups 
equipped with Hilbert-Schmidt distances by finite symmetric groups equipped with normalised Hamming distances.
The problem of providing analogous results to the one of Bowen and Burton in the Hilbert-Schmidt setting was
asked by G. Arzhansteva during the seminar 'Stability and Testability' at the IAS in November 2020.
The argument in \cite{BB} is based on a HNN-extension construction along different free subgroups of $\PSL_d (\mathbb Z)$, 
a completely different approach from ours.
Even though we exhibit other arithmetic groups (such as $\Sp_{2g}(\mathbb Z)$) which satisfy the analogous 
Hilbert-Schmidt statement, the case of $\PSL_d(\mathbb Z)$ is left open. It would be interesting to establish it nonetheless. 

\subsection*{Infinitely presented property (T) groups.}
Using the techniques established for the results above, we also prove the following:

\begin{thm} \label{thm:inf_presented}
Let $\Gamma$ be an infinitely presented property (T) group. If $\Gamma$ is weakly ucp-stable, then there exists
a non-hyperlinear group.
\end{thm}

For every prime $p$, an explicit example of an infinitely presented property
(T) group is $\SL_3(\mathbb F_p[X])$, where $\mathbb F_p[X]$ is the polynomial ring over the finite field $\mathbb F_p$.
Let us also mention that by a result of Gromov (Corollary 5.5.E in \cite{Gromov}), there are uncountably many such groups.
We refer to Section 3.4 in \cite{BDV} for more examples.
Unlike the two previous corollaries, the proof of Theorem \ref{thm:inf_presented} does not utilise cohomological
techniques. Instead, it invokes the theorem of Shalom \cite{Shalom} stating that $\Gamma$ will always be a 
quotient of a  finitely presented property (T) group.

\subsection*{\textbf{Non weakly ucp-stable property (T) groups}}
As mentioned before, Ioana, Spaas and Wiersma \cite{ISW} showed that $\SL_2(\mathbb Z) \ltimes \mathbb Z^2$ is
not  flexibly HS-stable. In fact, their proof shows that $\SL_2(\mathbb Z) \ltimes \mathbb Z^2$ is not
weakly ucp-stable. We claim that examples of groups with property (T) failing weak ucp-stability exist as well.
Indeed, the following proposition
was kindly pointed out to us by Andreas Thom:
\begin{prop} \label{prop:GR_analog}
Let $\Gamma$ be a hyperlinear property (T) group. If $\Gamma$ is weakly ucp-stable, then it is residually finite.
\end{prop}
An observation that appears in \cite{BL}, extending earlier observations of \cite{AP} and \cite{GR}, is
that the combination of hyperlinearity and flexible HS-stability implies residual finiteness. 
Proposition \ref{prop:GR_analog} shows that even though weak ucp-stability is strictly weaker than flexible HS-stability in general,
the observation of Becker and Lubotzky continues to hold for weak ucp-stability, under the assumption of property (T).
To this end, Thom \cite{Thom1} constructed examples of hyperlinear property (T) groups which are not residually
finite (in fact, some are not locally embeddable into amenable groups). By Proposition \ref{prop:GR_analog}, these
groups  will not be weakly ucp-stable.

\paragraph{\textbf{Outline}}
The article is organized as follows: In Section \ref{sec:prelim} we provide preliminaries on 
2-cohomology, property (T), stability, von Neumann algebras and ucp maps.
In Section \ref{sec:main} we prove Theorem \ref{thm:main2}.
In Section \ref{sec:lattices_and_rand_groups} and \ref{sec:inf_presented}, we prove Corollary \ref{cor:Sp}, Proposition \ref{prop:rand} and Theorem \ref{thm:inf_presented}.
Finally, we give the proof of Proposition \ref{prop:GR_analog} in Section \ref{sec:GR_analog}, and end with proving some cohomological lemmas 
in Section \ref{sec:coeff_change}.

\section*{Acknowledgments}
I would like to thank Adrian Ioana for his constant support and stimulating discussions which lead to the proof of
Theorem \ref{thm:ISW_gen}. I would also like to thank my advisors Uri Bader and Alex Lubotzky for their guidance and helpful suggestions. 
I am grateful for various conversations with Michael Chapman, Adam Dor-On, Michael Glasner, Andreas Thom, Itamar Vigdorovich and Geva Yashfe.
This work is part of the PhD thesis of the author, it was supported by the European Research Council (ERC) 
under the European Unions Horizon 2020 research and innovation program (Grant No. 882751).

\section{Preliminaries} \label{sec:prelim}

\subsection{2-cocycles, projective representations and property (T)} \label{prelim:cohom}

For background on group cohomology, we refer to \cite{Brown}. Nevertheless, we recall some definitions and basic facts. 
Given a group $\Gamma$ and an abelian group $A$ endowed with a trivial $\Gamma$ action,
the second cohomology group $H^2(\Gamma, A)$ can be described as follows: Let $Z^2(\Gamma, A)$ be the set of 
\emph{2-cocycles}, i.e. maps $c: \Gamma \times \Gamma \to A$ satisfying the 2-cocycle
identity:  $c(g,h)c(gh, k) = c(h,k)c(g,hk)$ for every $g,h,k\in\Gamma$.
A map $c: \Gamma\times\Gamma \to A$ of the form $c(g,h)=b(g)b(h)b(gh)^{-1}$, for some map 
$b: \Gamma \to A$, is said to be a \emph{2-coboundary}. Denote by $B^2(\Gamma, A)$ the set of 2-coboundaries. 
It is an easy computation to verify that $B^2(\Gamma, A) \subset Z^2(\Gamma, A)$ and that both are abelian groups. 
We can then define $H^2(\Gamma, A)$ to be the quotient $Z^2(\Gamma, A)/B^2(\Gamma,A)$.

The second cohomology group $H^2(\Gamma,A)$ classifies \emph{central}
extensions of the form:
\begin{center}
\begin{tikzcd}
1 \arrow[r] & A \arrow[r,"i"] & G \arrow[r,"q"] & \Gamma \arrow[r] & 1.
\end{tikzcd}
\end{center}

Indeed, given such a central extension, we can choose a (set-theoretic) section $\sigma: \Gamma \to G$ of the quotient map $q$ with $\sigma(1_{\Gamma}) = 1_G$.
For all $g,h \in \Gamma$, we can define $c(g,h) \in A$ to be the unique solution to the
equation $i(c(g,h)) = \sigma(g)\sigma(h)\sigma(gh)^{-1}$. It is readily verified that $c \in Z^2(\Gamma, A)$, and 
that the corresponding cohomology class $[c]$ in $H^2(\Gamma, A)$ is independent of the choice of $\sigma$.
Further, $[c]=0$ if and only if the extension splits (i.e. $\sigma$ can be chosen to be a homomorphism),
in this case, $G \simeq \Gamma \times A$.
Conversely, given a class $[c] \in H^2(\Gamma, A)$, one can always build a central extension that induces it as above.

Denote by $\mathbb T = U(1)$ the circle group. 
A \emph{projective representation} is a map $\pi: \Gamma \to U(\mathcal H)$, for some 
Hilbert space $\mathcal H$, such that for every $g,h \in \Gamma$ there is a scalar
$c(g,h) \in \mathbb T$ with $\pi(g) \pi(h) = c(g,h) \pi(gh)$. 
In this case $c: \Gamma \times \Gamma \to \mathbb T$ is a 2-cocycle, thus, it gives a class in $H^2(\Gamma, \mathbb T)$.
This class will be trivial if and only if there exists a unitary representation $\rho: \Gamma \to U(\mathcal H)$
and a map $b: \Gamma \to \mathbb T$ with $\pi(g) = b(g) \rho(g)$ for all $g \in \Gamma$.
Given a 2-cocycle $c \in Z^2(\Gamma, \mathbb T)$, one can in fact construct a projective representation $\lambda_c$ that
induces $c$ as above.
Indeed, define $\lambda_c: \Gamma \to U(\ell^2(\Gamma))$ to be the \emph{projective left regular representation} given by
$ \lambda_c(g)\delta_h = c(g,h)\delta_{gh}$
where $\delta_h \in \ell^2(\Gamma)$ denotes the dirac function supported on $h$. 
Then $\lambda_c(gh) = c(g,h)\lambda_c(g)\lambda_c(h)$, and $\lambda_c(1_\Gamma) = I$ if an only if $c(1_\Gamma,1_\Gamma)=1$.

For a definition and generalities on property (T), we refer to \cite{BDV}.
The following characterisation of property (T) in terms of projective representations will play a central role in the proof Theorem \ref{thm:ISW_gen}:

\begin{thm}[Lemma 1.1. in Nicoara, Popa and Sasyk \cite{NPS}]\label{thm:NPS}
If $\Gamma$ has property (T), then there exists a finite set $F\subset\Gamma$ and $\epsilon>0$ such that the following holds:

Let  $\pi:\Gamma\rightarrow \mathcal U(\mathcal H)$ be a projective representation with 2-cocycle $c:\Gamma\times\Gamma\rightarrow\mathbb T$. Assume there exists a unit vector $\xi\in\mathcal H$ with $\inf\{\|\pi(g)\xi-\alpha\xi\|\mid \alpha\in\mathbb T\}<\epsilon$, for every $g\in F$. 
Then there is a map $b:\Gamma\rightarrow\mathbb T$ and a vector $\eta\in \mathcal H$ such that  $c(h,k)=b(h)b(k)b(hk)^{-1}$ and $\pi(h)\eta=b(h)\eta$, for every $h,k\in\Gamma$. In particular, $c$ is a 2-coboundary.
\end{thm}

The theorem strengthens property (T) from the existence of invariant vectors for unitary representations with almost invariant vectors
to the existence of invariant vectors for \emph{projective} unitary representations with almost invariant vectors.
In case an invariant vector exists, the 2-cocycle associated with such a projective representation is in fact a 
2-coboundary. Thus, as explained above, the projective representation comes from a true unitary representation.

\subsection{Stability in terms of generators and relations} \label{prelim:stab_gen_rel}

Given a finitely generated group $\Gamma$, let us fix a presentation $\pi: \mathbb F_S \twoheadrightarrow \Gamma$, where 
$S = \{ g_1, \dots, g_k\}$ is a finite generating set of $\Gamma$ and $\mathbb F_S$ is the free group on $k$ generators. Denote by $\{\bar{s} \vert \; s \in S \}$ the generators of $\mathbb F_S$, so that $\pi(\bar{s}) = s$ for $s \in S$.
We shall usually fix such a presentation whenever dealing with a finitely generated group.
It will be convenient to work with the following consequence of weak ucp-stability (cf. Lemma 3.1 in \cite{Ioana1}):

\begin{lem} \label{lem:ucp_stab_defs}
Given a finitely generated group $\Gamma$ and a presentation as above, if $\Gamma$ is weakly ucp-stable,
then the following holds:
For any $\epsilon > 0$ there exists  $\delta > 0$ and finite subsets $R_0 \subset \ker \pi$, $W_0 \subset \mathbb F_S \setminus \ker \pi$ such that
for any finite dimensional Hilbert space $\mathcal H$ and any map $f: S \to U(\mathcal H)$ satisfying
\begin{align*}
&\Vert r(f(g_1), \dots, f(g_k)) - 1 \Vert_{2, \tau_{\dim \mathcal H}} < \delta \text{ for all $r \in R_0$,} \\
&\text{and } \Vert w(f(g_1), \dots, f(g_k)) - 1 \Vert_{2, \tau_{\dim \mathcal H}} > \sqrt{2} - \delta \text{ for all $w \in W_0$,}
\end{align*}
there exists a homomorphism $\rho: \Gamma \to U(\widehat{ \mathcal H})$, for some Hilbert space $\widehat{ \mathcal H}$ containing
$\mathcal H$ such that:
$$ \Vert f(s) - P \rho(s) P \Vert_{2, \tau_{\dim \mathcal H}} < \epsilon \text{ for all $s \in S$,}$$
where $P: \widehat{\mathcal H} \to \mathcal H$ is the projection onto $\mathcal H$.
\end{lem}

Given a map $f: S \to H$ for some group $H$, abusing notation, it is useful to extend it to a homomorphism defined on $\mathbb F_S$ by $f(r) = r(f(g_1), \dots, f(g_k))$ for $r \in \mathbb F_S$.

\begin{proof}
Let $R_n\subset \ker{\pi}, \;W_n \subset \mathbb F_S \setminus \ker \pi$ be increasing
sequences of finite sets with $\cup_nR_n=\ker(\pi)$, $\cup_n W_n=\mathbb F_S \setminus \ker \pi$. 
Let $p:\Gamma\rightarrow\mathbb F_S$ be a section of the quotient map $\pi$, so that $p(s)=\bar{s}$ for all $s\in S$.
Assume the condition prescribed in the Lemma fails, then there exists $\epsilon>0$ and maps $f_n: S\rightarrow U(\mathcal H_n)$,
with $\dim \mathcal H_n < \infty$, such that for all $n\in\mathbb N$:
$$\max\{\Vert f_n(r) - 1 \Vert_{2, \tau_{\dim \mathcal H}}\vert\; r\in R_n\} < \frac{1}{n}, \;\text{ }
 \min\{\Vert f_n(w) - 1 \Vert_{2, \tau_{\dim \mathcal H}}\vert\; w\in W_n\} > \sqrt{2} - \frac{1}{n}$$
and $\max\{\Vert f_n(s) - P_n\rho_n(s)P_n \Vert_{2, \tau_{\dim \mathcal H}}\vert \;  s \in S\}\geq\epsilon$ 
for any homomorphism $\rho_n:\Gamma\rightarrow U(\widehat{\mathcal H}_n)$, 
where $\widehat{\mathcal H}_n$ is a Hilbert space containing $\mathcal H_n$ and $P_n: \widehat{\mathcal{H}}_n \rightarrow \mathcal H_n$ the corresponding projection.
Define $\varphi_n:\Gamma\rightarrow U(\mathcal H_n)$ by  $\varphi_n(g)= f_n(p(g))$. It is then readily verified
that $\varphi_n$ is a hyperlinear approximation.
For all $n \in \mathbb N$, we have $\varphi_n(s)=f_n(p(s))=f_n(\bar s)=f_n(s)$ for every $s\in S$. Thus, we get that
$\max\{\Vert \varphi_n(s) - P_n\rho_n(s)P_n \Vert_{2, \tau_{\dim \mathcal H}}\vert\;  s \in S\}\geq\epsilon$, 
for all $\widehat{\mathcal{H}}_n$, $P_n$ and homomorphism $\rho_n:\Gamma\rightarrow U(\widehat{\mathcal H}_n)$ 
as above. This contradicts weak ucp-stability.
\end{proof}

The following is a folklore stability-type statement saying that a matrix which is close to
being unitary is in fact close to a true unitary (see for example  \cite{AD} for a proof).

\begin{lem}[Lemma 2.2. in \cite{AD} ] \label{lem:unitary_stab}
Let $M \in M_n(\mathbb{C})$. Then there exists a unitary $R \in U(n)$ such that $\Vert M - R \Vert_{2,\tau_n} \leq \Vert M^* M - 1 \Vert_{2,\tau_n}$.
\end{lem}

\subsection{von Neumann algebras and twisted group algebras} \label{prelim:von_neum}

In this article, we will mostly be dealing with tracial von Neumann algebras (and on one occasion, also 
semifinite traces). We refer to \cite{Popa}, \cite{Takesaki} for the general
theory. However, for the reader's convenience, we will remind the basic objects.

For a Hilbert space $\mathcal H$, let us denote by $B(\mathcal H)$ the algebra of all bounded linear operators on
$\mathcal H$. For $x \in B(\mathcal H)$, we will denote by $\Vert x \Vert$ its operator norm.
A \emph{$C^*$-algebra} is a $*$-closed subalgebra of $B(\mathcal H)$, for some Hilbert space $\mathcal H$, which is closed in the operator norm induced topology.
A \emph{von Neumann algebra} $\mathcal M$ is a unital $*$-closed subalgebra of $B(\mathcal{H})$, for some Hilbert
space $\mathcal H$, which is closed in the weak operator topology. Note that every von Neumann algebra is a C$^*$-algebra.
By von Neumann's double commutant theorem, a unital $*$-closed subalgebra $\mathcal M \subset \mathcal B(\mathcal H)$ is a 
von Neumann algebra if and only if $\mathcal M'' = \mathcal M$, where $\mathcal M'$ denotes the collection
of operators in $B(\mathcal H)$ which commute with $\mathcal M$. If $\mathcal H$ is separable, we will say $\mathcal M$ is separable.
A \emph{state} on $\mathcal M$ is a normalized positive linear functional $\varphi : \mathcal M \to \mathbb C$, meaning
$\varphi(x^*x) \geq 0$ for all $x \in \mathcal M$ and $\varphi(1_\mathcal M)=1$. 
A state $\varphi$ on $\mathcal M$ is said to be 
\begin{itemize}
    \item \emph{Faithful} if $\varphi(x^*x) = 0 \implies x = 0$ for all $x \in \mathcal M$.
    \item \emph{Normal} if it is continuous on the unit ball of $\mathcal M$, equipped with the
          weak operator topology.
    \item \emph{Tracial} if $\varphi(xy) = \varphi(yx)$ for all $x,y \in \mathcal M$.
\end{itemize}

A \emph{trace} $\tau$ on $\mathcal M$ is a state that satisfies all three properties listed. In this case,
the pair $(\mathcal M, \tau)$ is called a \emph{tracial von Neumann algebra}. 
A canonical example is $(M_m(\mathbb C), \tau_m)$, where $\tau_m$ the usual normalized trace on $M_m(\mathbb C)$.
The following example will be central in our constructions:
For a countable group $\Gamma$ and a 2-cocycle $c: \Gamma \times \Gamma \to \mathbb T$, we can form the 
\textbf{\emph{twisted group von Neumann algebra}} $L_c(\Gamma)$ by taking the weak operator topology closure of 
$\spam \{ \lambda_c(g) \vert g \in \Gamma \}$ in $B(\ell^2(\Gamma))$, $\lambda_c$ being the projective
left regular representation defined before. In this way, $L_c(\Gamma)$ becomes a separable tracial
von Neumann algebra with the regular trace $\tau(x) = \langle x \delta_{1_\Gamma}, \delta_{1_{\Gamma}} \rangle$.
We shall use from now on the alternative notation $u_g = \lambda_c(g)$ for the canonical unitaries defined by the
group in the algebra. Note that $\tau(u_g) = 0$ for all $1 \neq g \in \Gamma$.
For the trivial cocycle, we get the \textbf{\emph{group von Neumann algebra}} $L(\Gamma)$.

Given a tracial von Neumann algebra $(\mathcal M, \tau)$, we denote the associated inner product by $\langle x, y \rangle_\tau = \tau(y^*x)$, 
and the associated $2$-norm by  $\Vert x \Vert_{2,\tau} = \tau(x^*x)^{\frac{1}{2}}$. 
Define $L^2(\mathcal M, \tau)$ to be the Hilbert space completion of $\mathcal M$ in this inner product.
For $x \in \mathcal M$, we denote by $\widehat{x}$ the corresponding vector in $L^2(\mathcal M, \tau)$.
In this case, $L^2(\mathcal M, \tau)$ becomes a bimodule over $\mathcal M$. Explicitly, we can represent 
$\mathcal M, \mathcal M^{\text{op}}$ in $B(L^2(\mathcal M, \tau))$ by letting them act as follows: 
$x \cdot \widehat{y} \cdot z = \widehat{xyz}$ for $x,y,z \in \mathcal M$. This is called the \emph{standard representation of $\mathcal M$}.

Given a sequence of tracial von Neumann algebras $(\mathcal M_m, \tau_m)$, we can define a direct product by:
$$ \prod_m \mathcal M_m = \{ (x_m) \vert \; x_m \in \mathcal M_m, \; \sup_m \Vert x_m \Vert < \infty \},$$
that is, the direct product consists of all \emph{operator-norm} bounded sequences valued in $\mathcal M_m$.
The direct product will typically not be a tracial von Neumann algebra. To remedy this, we 
fix a nonprincipal ultrafilter $\mathcal U$ on $\mathbb N$, so that every bounded sequence 
$(\alpha_i)$ in $\mathbb C$ has an ultralimit, denoted by $\lim_{i \to \mathcal U} \alpha_i$ (see \cite{Pestov}).
Define the \textbf{\emph{tracial ultraproduct}} to be:

$$ \prod_{m \to \mathcal U} \mathcal M_m = \faktor{\prod_{m \in \mathbb N} \mathcal M_m}{\mathcal I},$$

where $\mathcal I = \{ (x_m)_m \vert \lim_{m \to \mathcal U} \Vert x_m \Vert_{2, \tau_m} = 0 \}$ consists of
the asymptotically \emph{2-norm} trivial sequences.
For an element $x \in \prod_{m \to \mathcal U} \mathcal M_m$, we will denote by $\{ x_m \}_m$ a choice of 
a representative sequence for $x$.
By a theorem of Sakai, the tracial ultraproduct is indeed tracial von Neumann algebra with the trace defined by
$\tau_{\mathcal U}(\{ x_m \}_m) = \lim_{m \to \mathcal U} \tau_m(x_m)$.

A separable tracial von Neumann algebra $(\mathcal M, \tau)$ is called \textbf{\emph{Connes embeddable}} if 
there exists a trace preserving normal $*$-homomorphism into an ultraproduct of matrix algebras 
$\prod_{m \to \mathcal U} M_m(\mathbb C)$. The choice of the ultrafilter does not effect the definition (see
Corollary 3.2 in \cite{HS_C*}). In this case, we can think of $\mathcal M$
as a (von Neumann) subalgebra of $\prod_{m \to \mathcal U} M_m(\mathbb C)$ with $\tau = \tau_\mathcal U \vert_{\mathcal M}$.
For more on Connes embeddable algebras, we refer to \cite{Ozawa}, \cite{Olesen}, \cite{CL}.

Radulescu \cite{Radulescu} proved that a countable group $\Gamma$ is hyperlinear if and only if $L(\Gamma)$ is Connes embeddable.
Further, Thom \cite{Thom1} gave the following criterion for Connes embeddability of twisted group algebras 
with cocycles coming from a fixed central extension:

\begin{lem}[Lemma 3.4 in Thom \cite{Thom1}] \label{lem:Thom}
Let $A$ be a countable abelian group, $\Gamma$ a countable group. Consider a central extension of the form
\begin{center}
\begin{tikzcd}
1 \arrow[r] & A \arrow[r] & G \arrow[r] & \Gamma \arrow[r] & 1
\end{tikzcd}
\end{center}
and a cocycle $ c  \in Z^2(\Gamma, A)$ representing this extension. Then $G$ is hyperlinear if
and only if for every character $\chi \in \widehat{A}$ the twisted group von Neumann algebra 
$L_{\chi \circ  c }(\Gamma)$ is Connes Embeddable.
\end{lem}

To end this section, we briefly recall the infinite dimensional amplification of a tracial von Neumann algebra:
Let $\mathcal M \subset B(\mathcal H)$ be a von Neumann algebra and let $\mathcal K$ be some Hilbert space. 
Define the tensor product $\mathcal N =\mathcal M \bar{\otimes} B(\mathcal K)$ to be 
the von Neumann algebra generated by the collection of operators $\{x \otimes y \vert x \in \mathcal M, y \in B(\mathcal K)\} \subset B(\mathcal H \otimes \mathcal K)$,
where $\mathcal H \otimes \mathcal K$ denotes the Hilbert space tensor product. 
This algebra can be viewed as block matrices acting on $\mathcal H \otimes \mathcal K$ with entries in $\mathcal M$
(see Section 5.1. in \cite{Popa}).

Let us denote by $\Tr$ the usual faithful normal \emph{semifinite} trace on $B(\mathcal K)$ 
(See Section 8.3. in \cite{Popa} for a definition of semifinite traces on von Neumann algebras).
If $(\mathcal M, \tau)$ is a tracial von Neumann algebra, then there exists a unique
faithful normal semifinite trace $\hat{\tau} = \tau \otimes \Tr$ on $\mathcal N$ that satisfies 
$\hat\tau(x \otimes y) = \tau(x) \cdot \Tr(y)$ for all $0 \leq x \in \mathcal M, 0\leq y \in B(\mathcal K)$.
In this situation (see \cite{PX}), one can define $L^2(\mathcal N, \hat \tau)$ to be the Hilbert space completion of the subspace $\mathcal S = \{x \in \mathcal N \vert \;\hat{\tau}(x^*x) < \infty \} \subset \mathcal N$,
equipped with the 2-norm $\Vert x \Vert_{2,\hat{\tau}} = \hat{\tau}(x^*x)^{\frac{1}{2}}$. 
Similarly to before, for $x \in \mathcal S$ we denote by $\widehat{x}$ the corresponding vector in $L^2(\mathcal N, \hat{\tau})$.
Further, $L^2(\mathcal N, \hat{\tau})$ comes equipped with a $\mathcal N$-bimodule structure given by $x \cdot \widehat{y} \cdot z = \widehat{xyz}$ for all $x,z
\in \mathcal N, y \in \mathcal S$.

\subsection{Completely positive maps and a dilation theorem} \label{prelim:ucp}

Next, we recall the basic definition of a unital completely positive map (abbreviated ucp map). We refer to
the Chapter 3 in the book \cite{Paulsen} for more on this topic.

Let $A$ and $B$ be unital C$^*$-algebras. Denote by $A^+ = \{ x^*x \vert x \in A\}$ the cone of positive elements 
in $A$, and by $U(A)$ the group of unitaries in $A$. Recall that for every $n \in \mathbb N$,
the matrix algebras $M_n(A) \simeq A \otimes M_n(\mathbb C)$ are also C$^*$-algebras.
Given a linear map $\varphi: A \to B$, we can consider the amplifications $\varphi^{(n)}: M_n(A) \to M_n(B)$ given by
$\varphi^{(n)}(\left[ a_{ij}\right]) = \left[\varphi(a_{ij}) \right]$.
If for all $n$ one has $\varphi^{(n)}(M_n(A)^+)\subset M_n(B)^+$ and $\varphi(1)=1$, then $\varphi$ is called a
\textbf{\emph{ucp map}}. Such maps are automatically contractive and $*$-preserving. A typical example of a ucp map is the compression of a $*$-representation: 
If $\pi: A \to B(\mathcal H)$ is a $*$-homomorphism, for some Hilbert space $\mathcal H$, and $P \in B(\mathcal H)$ is a projection,
then the map $\varphi: A \to B(P\cdot \mathcal H)$ defined by $\varphi(x) = P\pi(x)P$ is a ucp map. 

Let us now restrict to \emph{group $C^*$-algebras}: The \emph{\textbf{group $C^*$-algebra}}
$C^*(\Gamma)$ of a group $\Gamma$ is a unital C$^*$-algebra together an homomorphism
$g \in \Gamma \mapsto v_g \in U(C^*(\Gamma))$ that satisfies the following universal property:
For any unital C$^*$-algebra $B$ and any homomorphism $\pi: \Gamma \to U(B)$ there exists a unique $*$-homomorphism $\Pi: C^*(\Gamma) \to B$
with $\Pi(v_g) = \pi(g)$ for every $g \in \Gamma$. The group C$^*$-algebra is well defined up to isomorphism.
The following is an immediate consequence of this universal property and the above discussion:

\begin{lem} \label{lem:compression_2_ucp}
Let $\Gamma$ be a group and $\mathcal H$ be a Hilbert space. Let $\psi: \Gamma \to B(\mathcal H)$ be a map of the form
$\psi(g) = P \pi(g)P$, where $\pi: \Gamma \to U(\mathcal H)$ is a representation and $P \in B(\mathcal H)$ is a projection.
Then there is a ucp map $\Psi: C^*(\Gamma) \to B(P\cdot \mathcal H)$ such that $\Psi(v_g) = \psi(g)$ for all $g \in \Gamma$.
\footnote{Restriction of ucp maps $\Psi:C^*(\Gamma) \to B(\mathcal H)$ to $\Gamma$ are called \emph{operator-valued positive definite functions}.
One could formulate weak ucp-stability in these terms.}
\end{lem}
Conversely, Stinespring's dilation theorem (Chapter 4 in \cite{Paulsen}) implies that any ucp map $\Psi: C^*(\Gamma) \to B(\mathcal K)$, for a Hilbert space $\mathcal K$, is of this form. This explains the term ucp appearing in weak ucp-stability. 

The following structural result for ucp maps with values in a tracial von Neumann algebra will be crucial in the proof of Theorem \ref{thm:main2}.
It is a version of Stinespring's dilation theorem  which essentially appears in
\cite{Popa_corr} as well as \cite{Kasparov}. We will use a formulation due to Ioana, Spaas and Wiersma:

\begin{prop}[Proposition 4.2. in Ioana, Spaas Wiersma \cite{ISW}] \label{prop:dilation}
Let $A$ be a unital C$^*$-algebra, $(\mathcal M,\tau)$ be a tracial von Neumann algebra and $\psi:A\rightarrow \mathcal M$ be a ucp map.

Then there exists a Hilbert space $\mathcal H$, a rank one projection $P\in B(\mathcal H)$ and a $*$-homomorphism $\pi:A\rightarrow \mathcal M\bar{\otimes} B(\mathcal H)$
such that $\psi(a) \otimes P=(1\otimes P)\pi(a)(1\otimes P)$, for every $a\in A$.
\end{prop}

That is, every ucp map to a tracial von Neumann algebra $\mathcal M$ factors as a $*$-homomorphism to an amplification 
$\mathcal M \bar{\otimes} B(\mathcal H)$ followed by the compression by a "rank one projection over $\mathcal M$".

\section{Proof of Theorem \ref{thm:main2}} \label{sec:main}

We begin this section by proving a general criterion ensuring that a property (T) group is not weakly ucp-stable.
The criterion is similar to the one in Theorem A of \cite{ISW}, which involves the existence of \emph{finite 
dimensional} projective representations whose 2-cocycles are almost trivial. We drop the finite dimensionality
assumption, and merely assume \emph{Connes embeddability} of the associated von Neumann algebras.

\begin{thm} \label{thm:ISW_gen}
Let $\Gamma$ be a countable property $(T)$ group. Assume there exists a sequence of $2$-cocycles 
$c_n \in Z^2(\Gamma, \mathbb{T})$ with the following properties:

\begin{enumerate}
    \item For each $n$, the cohomology class $[c_n] \in H^2(\Gamma, \mathbb{T})$ is nontrivial.\label{ass:main_thm_1}
    \item For all $g,h \in \Gamma$, $c_n(g,h) \xrightarrow[n \to \infty]{} 1$. \label{ass:main_thm_2}
    \item For all $n$, the twisted group von Neumann algebra $L_{c_n}(\Gamma)$ is Connes embeddable.\label{ass:main_thm_3}
\end{enumerate}
Then $\Gamma$ is not weakly ucp-stable.
\end{thm}

Certain parts of the proof of Theorem \ref{thm:ISW_gen}, including the application of a dilation theorem 
(Proposition \ref{prop:dilation}) and Nicoara, Popa and Sasyk's characterisation of property (T) 
(Theorem \ref{thm:NPS}), appear verbatim in the proof of Theorem C in \cite{ISW}. However, we provide a full proof in this section for completeness. 

Throughout this section, we fix a nonprincipal ultrafilter $\mathcal U$ on $\mathbb N$
and denote by $\mathcal M = \prod_{m \to \mathcal U} M_m(\mathbb C)$ 
the ultraproduct of matrix algebras. Denote by $\tau_\mathcal M = \tau_\mathcal U$ the normalized trace on $\mathcal M$
(see Section \ref{prelim:von_neum}). 
Recall that for every $g \in \Gamma$, we denote by $v_g \in C^*(\Gamma)$
the corresponding unitary in the group C$^*$-algebra (see Section \ref{prelim:ucp}).
The following proposition strengthens weak ucp-stability from a rigidity property for hyperlinear approximations to a
rigidity property for "hyperlinear approximations" with values in the unitary group of $\mathcal M$.

\begin{prop} \label{prop:main}
Let $\Gamma$ be a finitely generated weakly ucp-stable group. Let $\varphi_n: \Gamma \to U(\mathcal M)$
be a sequence of maps satisfying:
\begin{enumerate}[label=(\roman*)]
    \item $\lim_n \Vert \varphi_n(g) \varphi_n(h) - \varphi_n(gh) \Vert_{2,\tau_\mathcal M} = 0$ for every $g,h\in
    \Gamma$. \label{eq:lem_1}
    \item $\liminf_n \Vert \varphi_n(g) - 1\Vert_{2,\tau_\mathcal M} \geq \sqrt{2}$ for every $1\neq g \in \Gamma$.
    \label{eq:lem_2}
\end{enumerate}
Then, for every $\epsilon > 0$ and finite subset $F \subset \Gamma$, for large enough $n \in \mathbb N$ there exist ucp maps $\psi_n: C^*(\Gamma) \to \mathcal M$ such that 
$\Vert \psi_n(v_s) - \varphi_n(s) \Vert_{2,\tau_\mathcal M} < \epsilon$ for all $s \in F$.
\end{prop}

\begin{proof}
Given $\epsilon>0$ and $F \subset \Gamma$, choose a finite generating set $S = \{g_1, \dots, g_k\} \subset \Gamma$ 
of $\Gamma$ with $F \subset S$. 
Furthermore, choose a presentation $\pi: \mathbb F_S \to \Gamma$, $\mathbb F_S$ being the free group on $k$-generators, as in Section \ref{prelim:stab_gen_rel}. 
Since $\Gamma$ is weakly ucp-stable, by a combination of Lemmas \ref{lem:ucp_stab_defs} and
\ref{lem:compression_2_ucp}, the following holds: 
There exists  $\delta > 0$ and finite subsets $R_0 \subset \ker \pi$, $W_0 \subset \mathbb F_S \setminus \ker \pi$ such that
for all $m \in \mathbb N$ and any map $f: S \to U(m)$ satisfying
\begin{align*}
&\Vert r(f(g_1), \dots, f(g_k)) - 1 \Vert_{2, \tau_{m}} < \delta \;\text{ for all $r \in R_0$,} \\
&\text{and } \Vert w(f(g_1), \dots, f(g_k)) - 1 \Vert_{2, \tau_{m}} > \sqrt{2} - \delta \;\text{ for all $w \in W_0$,}
\end{align*}
there exists a ucp map $\psi: C^*(\Gamma) \to M_m(\mathbb C)$ that satisfies 
$$ \Vert f(s) - \psi(v_s)  \Vert_{2, \tau_{m}} < \frac{\epsilon}{2}\text{ for all $s \in S$.}$$
For fixed $g \in \Gamma$ and $n \in \mathbb N$, we choose a representative sequence $\varphi_n(g) = \{ \varphi_{n,m}(g) \}_{m \to \mathcal U}$ where $\varphi_{n,m}(g) \in M_m(\mathbb C)$.
In fact, using Lemma \ref{lem:unitary_stab}, we can choose $\varphi_{n,m}(g) \in U(m)$ to be unitaries. 
As a consequence of assumptions \ref{eq:lem_1} and (\ref{eq:lem_2}, for large enough $n \in \mathbb N$ we have:
$$ \Vert r(\varphi_n(g_1), \dots, \varphi_n(g_k)) - 1_{\mathcal M} \Vert_{2, \tau_\mathcal M} < \frac{\delta}{2} 
\;\text{ for all $r \in R_0$,}$$
$$ 
\Vert w(\varphi_n(g_1), \dots, \varphi_n(g_k)) - 1_{\mathcal M} \Vert_{2, \tau_\mathcal M} 
>\sqrt{2} - \frac{\delta}{2} \;\text{ for all $w \in W_0$.}
$$
Fix such $n$, by the definition of $\tau_\mathcal M = \tau_\mathcal U$, for $\mathcal U$-almost every $m$ we have
\begin{align*}
    \Vert &r(\varphi_{n,m}(g_1), \dots, \varphi_{n,m}(g_k)) - 1_{M_m(\mathbb C)} \Vert_{2, \tau_m} < \delta,\\
    &\text{and } \Vert w(\varphi_{n,m}(g_1), \dots, \varphi_{n,m}(g_k)) - 1_{M_m(\mathbb C)} \Vert_{2, \tau_m} > \sqrt{2} - \delta,\\
\end{align*}
for all $r \in R_0$ and $w \in W_0$. As such, for $\mathcal U$-almost every $m$
we can choose ucp maps $\psi_{n,m}: C^*(\Gamma) \to M_m(\mathbb C)$ such that:
$$ \Vert \psi_{n,m}(v_s) - \varphi_{n,m}(s) \Vert_{2, \tau_m} < \frac{\epsilon}{2} \text{ for all $s \in S$}.$$
Define $\psi_n: C^*(\Gamma) \to \mathcal M$ by $\psi_n(a) = \{\psi_{n,m}(a)\}_m$
for every $a \in C^*(\Gamma)$. It is easily seen that as $\psi_{n,m}$ are ucp maps, so is $\psi_n$.
Further:
$$ \Vert \psi_n(v_{s}) - \varphi_n(s) \Vert_{2, \tau_\mathcal M} < \epsilon \text{ for all $s \in S$}.$$
Since $F \subset S$, we get the claim.
\end{proof}

Using Proposition \ref{prop:main}, we can now prove Theorem \ref{thm:ISW_gen}.

\begin{proof}[Proof of Theorem \ref{thm:ISW_gen}]
Assume by contradiction that $\Gamma$ is weakly ucp-stable. As $\Gamma$ has property $(T)$, we can choose $F, \epsilon$ as promised in Theorem \ref{thm:NPS}, and $\Gamma$ is finitely generated.
We may assume $c_n(1_\Gamma, 1_\Gamma) = 1$ for all $n$, otherwise, replace $c_n$ by $c_n(1_\Gamma, 1_\Gamma)^{-1}\cdot c_n$ and the conditions of the theorem still hold.

For $g \in \Gamma$ and $n \in \mathbb N$, denote by $u_{g,n} \in L_{c_n} (\Gamma)$ the corresponding canonical unitary that satisfies:
$$ u_{g,n} u_{h,n} = c_n(g,h) u_{gh, n} \text{ for all $h \in \Gamma$}. $$
As $c_n(1_\Gamma,1_\Gamma) = 1$, we have $u_{1_\Gamma, n} = 1$ for all $n \in \mathbb N$. 
By assumption (\ref{ass:main_thm_3}), for every $n \in \mathbb N$, $L_{c_n} (\Gamma)$ has a trace preserving normal $*$-homomorphism
into $\mathcal M = \prod_{m \to \mathcal U} M_m(\mathbb C)$.
To simplify notation, we fix such tracial embeddings and simply view $L_{c_n} (\Gamma) \subset \mathcal M$ as subalgebras.
Under this setup, by assumption (\ref{ass:main_thm_2}) we have:
\begin{equation*}
\lim_{n \to \infty} \Vert u_{g,n} u_{h,n} - u_{gh,n} \Vert_{2, \tau_\mathcal M} = 
\lim_{n\to \infty} \vert c_n(g,h) - 1 \vert = 0 \;\text{ for every $g,h \in \Gamma$},
\end{equation*}
as well as:
\begin{equation*}
\Vert u_{g,n} - 1 \Vert_{2, \tau_\mathcal M} =  (2 - 2\Ra{\tau_\mathcal{M} (u_{g,n})})^{1/2} = \sqrt{2} 
\;\text{ for every $n \in \mathbb N$ and $1 \neq g \in \Gamma$.}
\end{equation*}

As such, by Proposition \ref{prop:main}, for large enough $n \in \mathbb N$ there exists ucp maps
$\psi_n: C^*(\Gamma) \to \mathcal M$ such that $\Vert \psi_n(v_s) - u_{s,n} \Vert_{2,\tau_\mathcal M} < 
\epsilon^2/2$ for all $s \in F$. Fix such $n$, and let $c = c_n, u_g = u_{g,n}, \psi = \psi_n$, so $\psi: C^*(\Gamma) \to \mathcal M$ is a ucp map with
\begin{equation} \label{equation:closeness}
\Vert \psi(v_s) - u_s \Vert_{2,\tau_\mathcal M} < \frac{\epsilon^2}{2}\;\text{ for all $s \in F$.}
\end{equation}
By Proposition \ref{prop:dilation} there exists a Hilbert space $\mathcal H$,
a rank one projection $P \in \mathcal H$ and a $*$-homomorphism $\rho: C^*(\Gamma) \to \mathcal M \bar{\otimes} B(\mathcal H)$ such that
\begin{equation} \label{equation:compression}
    \psi(a) \otimes P = (1_\mathcal M \otimes P) \rho(a) (1_\mathcal M \otimes P) \;\text{ for all $a \in C^*(\Gamma)$.}
\end{equation}
Let $\mathcal N = \mathcal M \bar{\otimes} B(\mathcal H)$ and $\hat{\tau}_{\mathcal N} = \tau_\mathcal M \otimes \Tr$ be the semifinite trace on $\mathcal N$
discussed in Section \ref{prelim:von_neum}, and let $\widehat{\mathcal H} = L^2(\mathcal{N}, \hat{\tau}_{\mathcal{N}})$ the associated Hilbert space equipped with 
an $\mathcal N$-bimodule structure. 

Define $\mathcal K$ to be the subspace $\mathcal K = (1_\mathcal M \otimes P) \cdot \widehat{\mathcal H}$. Notice that $\mathcal K$ is invariant under 
the left multiplication action of $(1_\mathcal M \otimes P) \mathcal N (1_\mathcal M \otimes P) \simeq \mathcal M$, and under the right multiplication action
of $\mathcal N$. Thus, we can define a map $\pi: \Gamma \to U(\mathcal K)$ by
$$ \pi(g) \eta = (u_g \otimes P) \cdot \eta \cdot \rho(v_g)^* \; \text{for all $g \in \Gamma$, $\eta \in \mathcal K$.}$$
Using that $\rho$ is a $*$-homomorphism and $u_g u_h = c(g,h) u_{gh}$, one can verify that $\pi$ is a projective unitary representation with $\pi(g)\pi(h)=c(g,h)\pi(gh)$.
We claim that as a consequence of equations (\ref{equation:closeness}) and (\ref{equation:compression}),
the vector $\xi = \widehat{1_{\mathcal M} \otimes P} \in \mathcal K$ is a $(F, \epsilon)$-almost invariant unit vector of $\pi$. Indeed, 
$\Vert \xi \Vert_{2, \hat{\tau}_\mathcal N}^2 = \tau_\mathcal M(1_\mathcal M) \cdot \Tr(P) = 1$ and by equation (\ref{equation:compression}):
\begin{align*}
\langle \pi(g) \xi, \xi \rangle_{\hat{\tau}_\mathcal N} &= \langle (u_{g} \otimes P) \cdot \widehat{1_\mathcal{M} \otimes P} \cdot \rho(v_g)^* , \widehat {1_\mathcal M\otimes P} \rangle_{\hat{\tau}_\mathcal N} \\
&= \hat{\tau}_\mathcal N((u_{g} \otimes P) (1_\mathcal M \otimes P) \rho(v_g)^*  (1_\mathcal M \otimes P)) \\
&= \hat{\tau}_\mathcal N((u_g \otimes P) (\psi(v_g)^* \otimes P)) = \hat{\tau}_\mathcal N((u_g \psi(v_g)^*) \otimes P)\\
&= \tau_\mathcal M(u_g \psi(v_g)^*)
\end{align*}
for all $g \in \Gamma$. Thus, by the Cauchy-Schwarz inequality and equation (\ref{equation:closeness}) we have for all
$s \in F$:
\begin{align*}
\Vert \pi(s) \cdot \xi - \xi \Vert^2_{2, \hat{\tau}_\mathcal N} &= 2(1 - \Ra{\langle \pi(s) \cdot \xi, \xi \rangle_{\hat{\tau}_\mathcal N}}) \\
&= 2(1 - \Ra{ \tau_\mathcal M(u_{s} \psi(v_{s})^*) })\\ 
&= 2\Ra{ \tau_\mathcal M(u_{s}(u_{s}^* - \psi(v_{s})^*) })\\
&\underset{\text{C.S.}}{\leq} 2\Vert u_{s}^* - \psi(v_{s})^* \Vert_{2,\tau_\mathcal M} < \epsilon^2.
\end{align*}
By Theorem \ref{thm:NPS} we conclude that the cocycle $c$ is in fact a $2$-coboundary, contradicting assumption (\ref{ass:main_thm_1}). 
\end{proof}

Together with cohomological lemmas which will be proved in Section \ref{sec:coeff_change}, we can now deduce the main theorem.

\begin{proof}[Proof of Theorem \ref{thm:main2}]
Let us denote by $c \in Z^2(\Gamma, A)$ a cocycle induced by the given central extension.
Assume $G$ is hyperlinear, we wish to show $\Gamma$ is not weakly ucp-stable. 
By Lemma \ref{lem:Thom}, we have for each $\chi \in \widehat{A}$ that the twisted
group von Neumann algebra $L_{\chi \circ  c } (\Gamma)$ is Connes embeddable.
According to which extra assumption we imposed, that is, $\Gamma$ is perfect or that $G$ has property
(T) and $A = \mathbb Z$, we use one of Propositions \ref{prop:coeff_change_perfect}, \ref{prop:coeff_change_(T)}
to deduce that there exists some $\chi \in \widehat{A}$ such that
$[\chi \circ  c ] \neq 0 \in H^2(\Gamma, \mathbb{T})$. Note that $\{ \chi \vert \, [\chi \circ  c ] = 0 \in 
H^2(\Gamma, \mathbb{T})\}$ is a proper subgroup of $\widehat A$. 
Since $A$ is assumed to be torsion-free, $\widehat{A}$ will be connected  (see Theorem 30 in \cite{Morris}).
As a result, $\{ \chi \vert \, [\chi \circ  c ] = 0 \in  H^2(\Gamma, \mathbb{T})\}$ cannot contain an open 
neighborhood of the identity in $\widehat{A}$.
Thus, there exists a sequence $\chi_n \in \widehat A$ converging to the trivial character, but $[\chi_n \circ  c ] \neq 0$ for every $n$. 
Take $c_n = \chi_n \circ  c $ and we get a sequence satisfying the
conditions of Theorem \ref{thm:ISW_gen}. Consequently, $\Gamma$ is not weakly ucp-stable.
\end{proof}


%

\section{Results for lattices and random groups with property (T)}  \label{sec:lattices_and_rand_groups}

Recall in the setting of Corollary \ref{cor:Sp}, $G$ denotes a connected semisimple Lie group with infinite cyclic 
fundamental group and with property (T), $\Gamma \leq G$ is a lattice. Let $p: \widetilde{G} \to G$ denote the universal cover, and $\widetilde{\Gamma} = p^{-1}(\Gamma)$.

\begin{proof}[Proof of Corollary \ref{cor:Sp}]
Since $G$ is semisimple and has property (T), by Theorem 3.5.4 in \cite{BDV}, $\widetilde{G}$ also has property (T).
By Kazhdan's theorem, $\widetilde{\Gamma}$ has property (T), as it is a lattice in $\widetilde{G}$.
Since $\mathbb Z  = \ker (p) =\pi_1(G)$ is a central subgroup of $\widetilde{G}$, the central extension
\begin{center}
\begin{tikzcd}
1 \arrow[r] & \mathbb Z \arrow[r] & \widetilde{\Gamma} \arrow[r] & \Gamma \arrow[r] & 1
\end{tikzcd}
\end{center}

satisfies all the assumptions of Theorem \ref{thm:main2}, except that we need to check it is non-split. Indeed, if it
were, then $\widetilde{\Gamma} \simeq \Gamma \times \mathbb Z$, which cannot happen as property (T) for
$\widetilde{\Gamma}$ prevents it from having infinite abelian quotients.
\end{proof}

We now turn to the setting of Gromov random groups. For the relevant background regarding $K(\Gamma, 1)$ spaces, geometric dimension, as well as Euler characteristics, we refer to chapters VIII and IX in \cite{Brown}.

\begin{proof}[Proof of Proposition \ref{prop:rand}]

Let $\Gamma$ be a Gromov random group at density $1/3 < d < 1/2$ with $k$ generators and of length $3l$, presented 
by $\Gamma = \langle X \vert R \rangle$ as in Section \ref{sec:intro}.
By a theorem of Gromov (see Theorem 11 in \cite{Ollivier}), with overwhelming probability $\Gamma$ is a torsion-free hyperbolic group of geometric dimension 2,
meaning the presentation complex associated with $\Gamma = \langle X \vert R \rangle$ is 
a $K(\Gamma, 1)$ space. This immediately implies that $H_i(\Gamma, \mathbb Z) = 0$ for all $i > 2$. 
Further, the Euler characteristic of $\Gamma$ will be:
$$ \chi(\Gamma) = 1 - k + \vert R \vert \sim 1 - k + (2k-1)^{3ld}, $$
so for large $l$, $\chi(\Gamma)$ will certainly be greater than $2$.  

We claim that $\Gamma$ is perfect with overwhelming probability. Indeed, this follows immediately from the main
Theorem of Kozma and Lubotzky \cite{KL}, which says that with overwhelming probability $\Gamma$ cannot
be non-trivially represented in $\GL_n(\mathbb C)$ for fixed $n$:
any abelian quotient of $\Gamma$ is generated by $k$-elements, so it can be faithfully represented in
$\GL_k(\mathbb C)$. Thus, by the aforementioned result, this quotient should in fact be trivial 
(even $\mathbb Z/2\mathbb Z$ cannot appear, as we pick uneven lengths). As such, $H_1(\Gamma, \mathbb Z) = 0$ with overwhilming probability.
Note that by the Euler-Poincaré formula we have: 
$$2 \leq \chi(\Gamma) = 1 - \rank_\mathbb Z H_1(\Gamma, \mathbb Z) + \rank_\mathbb Z H_2(\Gamma, \mathbb Z) = 1 + \rank_\mathbb Z H_2(\Gamma, \mathbb Z),$$
where $\rank_\mathbb Z M$ denotes the rank of the free part of a finitely generated abelian group $M$. As such, the free part of $H_2(\Gamma, \mathbb Z)$ is not trivial.
Thus, together with the isomorphism $H^2(\Gamma, \mathbb Z) \simeq \Hom(H_2(\Gamma, \mathbb Z), \mathbb Z)$ given by the universal coefficient theorem, we have $H^2(\Gamma, \mathbb Z) \neq 0$.
Consequently, with overwhelming 
probability there exists a non-split central extension of the form:

\begin{center}
\begin{tikzcd}
1 \arrow[r] & \mathbb Z \arrow[r] & \widetilde{\Gamma} \arrow[r] & \Gamma \arrow[r] & 1.
\end{tikzcd}
\end{center}
By a seminal result due to \.{Z}uk \cite{Zuk} which was expanded upon in \cite{KK}, $\Gamma$ will have property (T) 
with overwhelming probability. Thus, the assumptions of Theorem \ref{thm:main2} are satisfied with overwhelming probability.
\end{proof}

Let us comment on a small technicality: we required the lengths of relations to be divisible by $3$ as a shadow of the triangle model for random groups,
in which property (T) is usually proved (see \cite{Ashcroft}, \cite{KK}, \cite{Zuk}). However, this can be remedied if we allow to randomly
pick relations of length \emph{close} to $l$ (see the discussion in I.2.c in \cite{Ollivier}, see also the remark on page 3 of \cite{Ashcroft}).
Proposition \ref{prop:rand} holds in this model as well.

\section{Infinitely presented property (T) groups} \label{sec:inf_presented}
The proof of Theorem \ref{thm:inf_presented} is based on another application of Proposition \ref{prop:main}
and proceeds along the lines of the proof of Theorem G in \cite{ISW}.
As in Section \ref{sec:main}, fix a nonprincipal ultrafilter $\mathcal U$ on $\mathbb N$
and denote by $\mathcal M = \prod_{m \to \mathcal U} M_m(\mathbb C)$ 
the ultraproduct of matrix algebras, $\tau_\mathcal M = \tau_\mathcal U$ the normalized trace on $\mathcal M$. 
For a group $G$, for every $g \in G$, $v_g \in C^*(G)$ denotes the corresponding unitary in the group C$^*$-algebra.

\begin{proof}[Proof or Theorem \ref{thm:inf_presented}]
Let $\Gamma$ be an infinitely presented property (T) group. Assume $\Gamma$ is weakly ucp-stable, and assume by
contradiction that all groups are hyperlinear. 
As $\Gamma$ has property (T), it is finitely generated.
Fix a presentation $\Gamma = \langle s_1, \dots, s_k \vert r_l, l\in \mathbb N\rangle$. By a theorem of Shalom, 
there exists some $m \in \mathbb N$ such that $\Gamma_0 = \langle s_1, \dots, s_k \vert r_1, \dots r_m \rangle$
has property (T) (see Theorem 6.7 in \cite{Shalom}, as well as Theorem 3.4.5 in \cite{BDV}).
As a consequence of property (T), there exists a finite subset $F \subset \Gamma_0$ and $\epsilon>0$ such that if
$\pi: \Gamma_0 \to U(\mathcal H)$ is a unitary representation of $\Gamma_0$ on a Hilbert space $\mathcal H$ and
$\xi  \in \mathcal H$ is a unit vector with $\Vert \pi(s)\xi - \xi \Vert \leq \epsilon$ for all $s \in S$, then 
$\Vert \pi(g) \xi - \xi \Vert < 1$ for all $g \in \Gamma_0$ (see Proposition 1.1.9 in \cite{BDV}).

Denote by $p: \Gamma_0 \to \Gamma$ the quotient map and fix a section $\sigma: \Gamma \to \Gamma_0$ of $p$ with
$\sigma(1) = 1$. 
For $n \geq 0$, let $\Gamma_n = \langle s_1, 
\dots, s_k \vert r_1, \dots r_{m+n} \rangle$ and let $p_n: \Gamma _0 \to \Gamma_n$ be the quotient map.
For $g \in \Gamma_n$, denote by $u_{g,n} \in L(\Gamma_n)$ the corresponding unitary in the group von Neumann algebra.
By assumption, each $\Gamma_n$
is hyperlinear, hence, by Radulescu's theorem \cite{Radulescu} $L(\Gamma_n)$ is Connes embeddable. 
As such, for each $n$ we can fix a trace preserving normal $*$-homomorphism $L(\Gamma_n) \hookrightarrow\mathcal M 
= \prod_{m \to \mathcal U} M_m(\mathbb C)$, and view $L(\Gamma_n) \subset \mathcal M$ as subalgebras.

For each $n \in \mathbb N$, define $\varphi_n: \Gamma \to U(\mathcal M)$ by $\varphi_n(g) = u_{p_n(\sigma(g)), n}$
for all $g \in \Gamma$. We claim $\varphi_n$ satisfies the assumptions of Proposition \ref{prop:main}. Indeed,
note that $\ker p = \bigcup_n \ker p_n$ as an increasing union. As such, for all $g,h \in \Gamma$, for large 
enough $n$ we have $\sigma(g) \sigma(h) \sigma(gh)^{-1} \in \ker p_n$, so $\varphi_n(g) \varphi_n(h) = \varphi_n(gh)$.
Furthermore, for every $1 \neq g \in \Gamma$ and $n \in \mathbb N$ we have $\sigma(g) \notin \ker p_n$, so
$\Vert \varphi_n(g) - 1\Vert_{2,\tau_\mathcal M} = (2 - 2\Ra \tau_{\mathcal M} (u_{p_n(\sigma(g)), n}))^{1/2} = \sqrt{2}$.
By Proposition \ref{prop:main}, for large enough $n$ there exist ucp maps $\psi_n: C^*(\Gamma) \to \mathcal M$ such that 
$\Vert \psi_n(v_{p(s)}) - \varphi_n(p(s)) \Vert_{2,\tau_\mathcal M} < \epsilon^2/2$ for all $s \in F$. 
For each $s \in F$, for large enough $n$ we have $\sigma(p(s))s^{-1} \in \ker p_n$, so 
$\varphi_n(p(s)) = u_{p_n(s),n}$. Thus, for large enough $n$ we have:
$$\Vert \psi_n(v_{p(s)}) - u_{p_n(s),n} \Vert_{2,\tau_\mathcal M} < \frac{\epsilon^2}{2} \text{ for all $s \in F$.}$$ 

Fix such $n$, let $u_g = u_{p_n(g),n}$ for all $g \in \Gamma_0$ and let $\psi = \psi_n \circ \bar{p}$, 
where $\bar{p}: C^*(\Gamma_0) \to C^*(\Gamma)$ is the $*$-homomorphism induced by $p$.
Thus, $\psi: C^*(\Gamma_0) \to \mathcal M$ is a ucp map with
\begin{equation*}
\Vert \psi(v_s) - u_s \Vert_{2,\tau_\mathcal M} < \frac{\epsilon^2}{2}\;\text{ for all $s \in F$.}
\end{equation*}
We can now repeat the dilation argument in the proof of Theorem \ref{thm:ISW_gen}, with $\Gamma_0$ in place of 
$\Gamma$, to find a Hilbert space
$\mathcal K$, a  unitary representation $\pi: \Gamma_0 \to U(\mathcal K)$ and a unit vector
$\xi \in \mathcal K$ such that 
$$
\langle \pi(g) \xi, \xi \rangle_{\mathcal K}  = \tau_\mathcal M(u_g \psi(v_g)^*) \text{ for all $g \in \Gamma_0$,}
$$
and 
$$ 
\Vert \pi(s) \xi - \xi \Vert^2_{\mathcal K} \leq 2\Vert u_{s}^* - \psi(v_{s})^* \Vert_{2,\tau_\mathcal M} < \epsilon^2 \text{ for all $s \in F$.}
$$
As a consequence, we must have $\Vert \pi(g) \xi - \xi \Vert_{\mathcal K} < 1$ for all $g \in \Gamma_0$.
Since $\psi$ is contractive, we have for all $g \in \Gamma_0$:
\begin{align*}
\Vert u_g - \psi(v_g) \Vert_{2,\tau_\mathcal M}^2 &= 1 + \Vert \psi(v_g) \Vert_{2,\tau_\mathcal M}^2 - 2\Ra{ \tau_\mathcal M(u_{g} \psi(v_{g})^*) }\\
&\leq 2(1 - \Ra{ \tau_\mathcal M(u_{g} \psi(v_{g})^*) })\\ 
&= 2(1 - \Ra{\langle \pi(g) \cdot \xi, \xi \rangle_{\mathcal K}}) =
\Vert \pi(g) \xi - \xi \Vert^2_{\mathcal K}< 1.
\end{align*} 
Now, for each $g \in \ker p$ we have $\psi(v_g) = \psi_n(v_{p(g)}) = 1$. On the other hand, if
$g \notin \ker p_n$, then $\Vert u_g - 1 \Vert_{2, \tau_\mathcal M} = \Vert u_{p_n(g), n} - 1 \Vert_{2, \mathcal M}
= \sqrt{2}$. Consequently, we have $\ker p = \ker p_n$ and $\Gamma$ is finitely presented. Contradiction.
\end{proof}

\section{Weak ucp-stability and residual finiteness} \label{sec:GR_analog}
We present the proof of Proposition \ref{prop:GR_analog}. For this, recall that a group $\Gamma$ has
\emph{Kirchberg's Factorization Property} if there exists a sequence $d_n \in \mathbb N$ and a sequence
of ucp maps $\psi_n : C^*(\Gamma) \to M_{d_n}(\mathbb C)$ such that:
\begin{enumerate}
    \item $\lim_{n \to \infty} \Vert \psi_n(a)\psi_n(b) - \psi_n(ab) \Vert_{2,\tau_{d_n}} =0$ for all $a,b\in C^*(\Gamma)$, \label{item:mult}
    \item $\lim_{n \to \infty} \tau_{d_n}(\psi_n(a)) = \tau(a)$ for every $a \in C^*(\Gamma)$, \label{item:trace}
\end{enumerate}
where $\tau: C^*(\Gamma) \to \mathbb C$ is the regular trace. 
A celebrated result of Kirchberg \cite{Kirchberg} states that once a group $\Gamma$ has property (T), it has the
factorization property if and only if it is residually finite. The proof of Proposition \ref{prop:GR_analog} follows:
\begin{proof}[Proof of Proposition \ref{prop:GR_analog}]
Let $\Gamma$ be a hyperlinear property (T) group. It is well known that we can choose a hyperlinear approximation
$\varphi_n : \Gamma \to U(d_n)$ with $\lim_{n} \tau_{d_n}(\varphi_n(g)) = 0$
for every $1\neq g \in \Gamma$ (see for example Proposition II.2.9 in \cite{CL}).
By weak ucp-stability and Lemma \ref{lem:compression_2_ucp}, there exists a sequence of ucp maps
$\psi_n: C^*(\Gamma) \to M_{d_n}(\mathbb C)$ with
\begin{equation} \label{eq:app1}
\lim_{n \to \infty} \Vert \varphi_n(g) - \psi_n(v_g) \Vert_{2, \tau_{d_n}} = 0 \text{ for every $g \in \Gamma$,}
\end{equation}
where $v_g \in C^*(\Gamma)$ denotes the unitary corresponding to $g \in \Gamma$. Since $\varphi_n$ is an asymptotic
homomorphism, we see that $\psi_n$ satisfies $\lim_{n} \Vert \psi_n(v_g)\psi_n(v_h) - \psi_n(v_{gh})
\Vert_{2,\tau_{d_n}} =0$ for all $g,h\in \Gamma$. Further, for every $g  \in \Gamma$:
$$\vert \tau_{d_n} (\varphi_n(g)) - \tau_{d_n}(\psi_n(v_g)) \vert \leq \Vert \varphi_n(g) - \psi_n(v_g) \Vert_{2,\tau_{d_n}} \xrightarrow[n \to \infty]{} 0,$$
as such, $\lim_n \tau_{d_n} (\psi_n(v_g)) = 0 = \tau(v_g)$ for all $1 \neq g \in \Gamma$. 
Since $\{v_g \vert g \in \Gamma \}$ spans a dense subspace of $C^*(\Gamma)$, it follows that $\psi_n$ exhibits
$\Gamma$ has Kirchberg's factorization property. By Kirchberg's result \cite{Kirchberg}, $\Gamma$ must be residually finite.
\end{proof}

\section{Passing to $\mathbb T$-valued cocycles} \label{sec:coeff_change}
In this section, we prove the propositions used in Section \ref{sec:main} for passing from non-split central
extension to projective representations with 2-cocycles which are not 2-coboundaries. More precisely:
let $c \in Z^2(\Gamma, K)$ be a 2-cocycle, where $K$ is some abelian group. Assuming $c$ is not a coboundary,
we would like to find characters $\chi \in  \widehat{K}$ such that $\chi \circ c \in Z^2(\Gamma, \mathbb T)$ is also not a coboundary. 
Of course, this is not always possible, as the example of $\Gamma = \mathbb Z / n$, $K = \mathbb Z$, $c$ corresponding to the following extension shows:
\begin{center}
\begin{tikzcd}
1 \arrow[r] & \mathbb Z \arrow[r, "\cdot n"] & \mathbb Z \arrow[r] & \mathbb Z / n \arrow[r] & 1.
\end{tikzcd}
\end{center}
We give two situations where this can be guaranteed. Both these statements are derived from
arguments contained in the proofs of Theorem 6.1 and Lemma 6.4 in \cite{ISW} (see also the correction of 
\cite{ISW}). However, there are some subtleties
regarding assumptions on $\Gamma$ and $A$, so we spell out the proofs.

\begin{prop} \label{prop:coeff_change_perfect}
Let $K$ be a countable abelian group and let $\Gamma$ be a countable perfect group. Given a 2-cocycle $c \in Z^2(\Gamma, K)$ which is not a 2-coboundary, there exists a character $\chi \in \widehat{K}$ such that $\chi \circ c \in Z^2(\Gamma, \mathbb{T})$ is not a 2-coboundary.
\end{prop}

\begin{proof}
Let us equip $\mathbb T^{\Gamma^n}, n \in \mathbb N$ with the topology of pointwise convergence, making them abelian polish groups.
Assume by contradiction the conclusion is false, that is, the image of the map $\widehat{K} \ni \chi \mapsto 
c_\chi = \chi \circ c \in Z^2(\Gamma, \mathbb{T})$ is contained in $B^2(\Gamma, \mathbb{T})$. Note that this
map is continuous, and using basic facts from descriptive set theory, we can choose a Borel section for 
the coboundary map $\partial^1: \mathbb T^\Gamma \to B^2(\Gamma, \mathbb T) \subset \mathbb T^{\Gamma \times \Gamma}$ (see Theorems 1.4, 2.5 and 2.6 in \cite{Kechris}).
Thus, we can find a Borel map $\widehat{K} \ni \chi \to b_\chi \in \mathbb{T}^\Gamma$ such that:
$$ c_\chi (g,h) = b_\chi(g) b_\chi(h) b_\chi (gh)^{-1} = \partial^1(b_\chi)(g,h) \text{ for every $\chi \in \widehat{K}, g,h\in\Gamma$. }$$
Define $d_{\chi, \eta} = b_\chi b_\eta b_{\chi \eta}^{-1}$ for $\chi,\eta \in \widehat K$. We claim $d_{\chi, \eta} : \Gamma \to \mathbb{T} $ is in fact a homomorphism for every $\chi, \eta \in \widehat{K}$. Indeed, using the
fact that for fixed $g,h \in \Gamma$, the map $\chi \in \widehat K \mapsto c_\chi (g,h) \in \mathbb{T}$ is a homomorphism:

\begin{align*}
    d_{\chi, \eta} (g) d_{\chi, \eta} (h) &= b_\chi(g) b_\eta(g) b_{\chi \eta}(g)^{-1} b_\chi(h) b_\eta(h) b_{\chi \eta}(h)^{-1} \\
    &= (b_\chi(g) b_\chi(h) b_{\chi}(gh)^{-1}) (b_\eta(g) b_\eta(h) b_{\eta}(gh)^{-1}) b_\chi(gh) b_\eta(gh) b_{\chi\eta}(g)^{-1} b_{\chi\eta}(h)^{-1} \\
    &= c_\chi(g,h) c_\eta(g,h) b_\chi(gh) b_\eta(gh) b_{\chi\eta}(g)^{-1} b_{\chi\eta}(h)^{-1} \\
    &= c_{\chi\eta}(g,h) b_\chi(gh) b_\eta(gh) b_{\chi\eta}(g)^{-1} b_{\chi\eta}(h)^{-1} \\
    &= b_{\chi\eta}(g) b_{\chi\eta}(h) b_{\chi \eta}(gh)^{-1} b_\chi(gh) b_\eta(gh) b_{\chi\eta}(g)^{-1} b_{\chi\eta}(h)^{-1} \\
    &= b_\chi(gh) b_\eta(gh) b_{\chi \eta}(gh)^{-1}    \\
    &= d_{\chi, \eta}(gh)
\end{align*}
(alternatively, applying the coboundary operator $\partial ^1$ to $d_{\chi, \eta}$ shows it is in $Z^1(\Gamma, \mathbb{T})$, which shows it is homomorphism). 
Since $\Gamma$ is perfect, $d_{\chi, \eta}$ is trivial. Thus, for all $g \in \Gamma, \;\chi,\eta \in \widehat K$ we have
$b_\chi(g) b_\eta(g) = b_{\chi \eta}(g)$. As such, for a fixed $g \in \Gamma$ the map $\widehat{K} \ni \chi 
\to b_\chi(g) \in \mathbb T$ is a Borel measurable homomorphism.
By automatic continuity (see Theorem 2.2 in \cite{Rosendal}), this map is in fact a continuous character of $\widehat{K}$.
By Pontryagin duality, for $g \in \Gamma$ there exists $k_g \in K$ such that $\chi(k_g) = b_\chi(g)$ for all $\chi \in \widehat K$. As a result:
$$ c_\chi(g,h) = b_\chi(g) b_\chi(h) b_\chi(gh)^{-1} = \chi(k_g k_h k_{gh}^{-1}) \text{ for all $g,h \in \Gamma, \; \chi \in \widehat K$,}$$
as such, $c (g,h) = k_g k_h k_{gh}^{-1}$ is in fact a coboundary. Contradiction.
\end{proof}

\begin{prop}[cf. Lemma 6.4 in \cite{ISW}] \label{prop:coeff_change_(T)}
Let $\Gamma$ be a countable group. Let $c \in Z^2(\Gamma, \mathbb Z)$ be a 2-cocycle coming from choosing
a section for a central extension of the form:
\begin{center}
\begin{tikzcd}
1 \arrow[r] & \mathbb Z \arrow[r] & \widetilde{\Gamma} \arrow[r, "p"] & \Gamma \arrow[r] & 1.
\end{tikzcd}
\end{center}
If $\widetilde{\Gamma}$ has property (T), then there exists a character $\chi \in \widehat{\mathbb Z}$ such that
$\chi \circ c \in Z^2(\Gamma, \mathbb{T})$ is not a 2-coboundary.
\end{prop}

\begin{proof}
Let $\sigma: \Gamma \to \widetilde{\Gamma}$ be a section of the quotient map $p$ with $\sigma(g)\sigma(h)\sigma(gh)^{-1} = c(g,h)$ for all $g,h \in \Gamma$.

We claim that viewing $c$ as an $\mathbb R$ valued cocycle, it is not in $B^2(\Gamma, \mathbb R)$.
Assume by contradiction that there exists a map $b: \Gamma \to \mathbb R$ such that $c(g,h) = b(g)b(h)b(gh)^{-1}$
for all $g,h \in \Gamma$ (with multiplicative notation for $+$ in $\mathbb R$).
Define $r: \widetilde{\Gamma} \to \mathbb Z$ by $r(g) = \sigma(p(g))g^{-1}$ for $g \in \widetilde{\Gamma}$. 
As $\mathbb Z$ is central in $\widetilde{\Gamma}$, we have $r(g)r(h)=c(p(g),p(h))r(gh)$ for every $g,h\in\widetilde{\Gamma}$.
Hence, the map $s: \widetilde{\Gamma} \to \mathbb R$ given by $s(g)=b(p(g))^{-1}r(g)$ is a homomorphism.
Since $\widetilde{\Gamma}$ has property (T), so does $s(\widetilde{\Gamma})$, hence it is finite.
As $\mathbb R$ has no nontrivial finite subgroups, the image
of $s$ must be trivial. In particular, for all $g \in \mathbb Z \subset \widetilde{\Gamma}$ we have
$0 = s(g) = b(p(g))^{-1}\sigma(p(g))g^{-1} = b(1_\Gamma)^{-1} \sigma(1_\Gamma) g^{-1}$. 
Thus, $g = b(1_\Gamma)^{-1} \sigma(1_\Gamma)$ for all $g \in \mathbb Z$, which is a clear contradiction. 

As such, $c \in Z^2(\Gamma, \mathbb R)$ is not a 2-coboundary. By Corollary 6.2 in \cite{ISW}, there exists 
$\xi \in \widehat{\mathbb R}$ such that $\xi \circ c \in Z^2(\Gamma, \mathbb T)$ is not a $2$-coboundary.
Defining $\chi = \xi \vert _{\mathbb Z}$, we get a character of $\mathbb Z$ with 
$\chi \circ c \notin B^2(\Gamma, \mathbb T)$.

\end{proof}



\begin{bibdiv}
\begin{biblist}

\bib{AD}{article}{
   author={Akhtiamov, D.},
   author={Dogon, A.},
   title={On uniform Hilbert Schmidt stability of groups},
   journal={Proc. Amer. Math. Soc.},
   volume={150},
   date={2022},
   number={4},
   pages={1799--1809},
   issn={0002-9939},
   review={\MR{4375766}},
   doi={10.1090/proc/15772},
}

\bib{AP}{article}{
   author={Arzhantseva, G.},
   author={P\u{a}unescu, L.},
   title={Almost commuting permutations are near commuting permutations},
   journal={J. Funct. Anal.},
   volume={269},
   date={2015},
   number={3},
   pages={745--757},
   issn={0022-1236},
   review={\MR{3350728}},
   doi={10.1016/j.jfa.2015.02.013},
}

\bib{Ashcroft}{article}{
  title={Property (T) in density-type models of random groups},
  author={C. J. Ashcroft},
  journal={arXiv preprint arXiv:2104.14986},
  year={2021}
}


\bib{AK}{article}{
   author={Atkinson, S.},
   author={Kunnawalkam Elayavalli, S.},
   title={On ultraproduct embeddings and amenability for tracial von Neumann
   algebras},
   journal={Int. Math. Res. Not. IMRN},
   date={2021},
   number={4},
   pages={2882--2918},
   issn={1073-7928},
   review={\MR{4218341}},
   doi={10.1093/imrn/rnaa257},
}


\bib{BLSW}{article}{
    title={Stability and instability of lattices in semisimple groups},
    author={Bader, U.},
    author={Lubotzky, A.},
    author={Sauer, R.},
    author={Weinberger, S.},
    journal={arXiv preprint arXiv:2303.08943},
    year={2023}
}


\bib{Popa}{article}{
  title={An introduction to $II_1$ factors},
  author={C. Anantharaman},
  author={S. Popa},
  eprint={https://www.math.ucla.edu/~popa/Books/IIun.pdf},
  date={preprint 2010}
}


\bib{BL}{article}{
   author={O. Becker},
   author={A. Lubotzky},
   title={Group stability and Property (T)},
   journal={J. Funct. Anal.},
   volume={278},
   date={2020},
   number={1},
   pages={108298, 20},
   issn={0022-1236},
   review={\MR{4027744}},
   doi={10.1016/j.jfa.2019.108298},
}

\bib{BDV}{book}{
  title={Kazhdan's Property (T)},
  author={B. Bekka},
  author={P. de la Harpe},
  author={A. Valette},
  isbn={9780521887205},
  series={New Mathematical Monographs},
  year={2008},
  publisher={Cambridge University Press}
}


\bib{BB}{article}{
   author={L. Bowen},
   author={P. Burton},
   title={Flexible stability and nonsoficity},
   journal={Trans. Amer. Math. Soc.},
   volume={373},
   date={2020},
   number={6},
   pages={4469--4481},
   issn={0002-9947},
   review={\MR{4105530}},
   doi={10.1090/tran/8047},
}

\bib{Bowen}{article}{
  title={A brief introduction to sofic entropy theory},
  author={Bowen, L. },
  booktitle={Proceedings of the International Congress of Mathematicians (ICM 2018) (In 4 Volumes) Proceedings of the International Congress of Mathematicians 2018},
  pages={1847--1866},
  year={2018},
  organization={World Scientific}
}


\bib{Brown}{book}{
  title={Cohomology of Groups},
  author={K. S. Brown},
  isbn={9780387906881},
  series={Graduate Texts in Mathematics},
  year={1994},
  publisher={Springer New York}
}

\bib{Burton}{article}{
  title={Hyperlinear approximations to amenable groups come from sofic approximations},
  author={Burton, P.},
  journal={arXiv preprint arXiv:2110.03076},
  year={2021}
}


\bib{CL}{book}{
  title={Introduction to sofic and hyperlinear groups and Connes' embedding conjecture},
  author={V. Capraro},
  author={M. Lupini},
  author={V. Pestov},
  volume={1},
  year={2015},
  publisher={Springer}
}


\bib{CE}{article}{
   author={Choi, M. D.},
   author={Effros, E. G.},
   title={The completely positive lifting problem for $C\sp*$-algebras},
   journal={Ann. of Math. (2)},
   volume={104},
   date={1976},
   number={3},
   pages={585--609},
   issn={0003-486X},
   review={\MR{417795}},
   doi={10.2307/1970968},
}

\bib{Connes}{article}{
 author = {A. Connes},
 journal = {Annals of Mathematics},
 number = {1},
 pages = {73--115},
 publisher = {Annals of Mathematics},
 title = {Classification of Injective Factors},
 volume = {104},
 year = {1976}
}


\bib{Deligne}{article}{
   author={P. Deligne},
   title={Extensions centrales non r\'{e}siduellement finies de groupes
   arithm\'{e}tiques},
   language={French, with English summary},
   journal={C. R. Acad. Sci. Paris S\'{e}r. A-B},
   volume={287},
   date={1978},
   number={4},
   pages={A203--A208},
   issn={0151-0509},
   review={\MR{507760}},
}

\bib{DGLT}{article}{
  title={Stability, cohomology vanishing, and nonapproximable groups},
  author={M. De Chiffre},
  author={L. Glebsky},
  author={A. Lubotzky},
  author={A. Thom},
  booktitle={Forum of Mathematics, Sigma},
  volume={8},
  year={2020},
  organization={Cambridge University Press}
}

\bib{delaSalle}{article}{
  title={Spectral gap and stability for groups and non-local games},
  author={M. de la Salle},
  journal={arXiv preprint arXiv:2204.07084},
  year={2022}
}


\bib{EckShul}{article}{
  title={On amenable Hilbert-Schmidt stable groups},
  author={C. Eckhardt},
  author={T. Shulman},
  journal={arXiv preprint arXiv:2207.01089},
  year={2022}
}


\bib{ES1}{article}{
  title={Sofic groups and direct finiteness},
  author={G. Elek and E. Szab{\'o}},
  journal={Journal of Algebra},
  year={2003},
  volume={280},
  pages={426-434}
}

\bib{ES2}{article}{
  title={Hyperlinearity, essentially free actions and $L_2$-invariants. The sofic property},
  author={G. Elek and E. Szab{\'o}},
  journal={Mathematische Annalen},
  year={2004},
  volume={332},
  pages={421-441}
}

\bib{GS}{article}{
  title={Virtually free groups are $ p $-Schatten stable},
  author={M. Gerasimova},
  author={K. Shchepin},
  journal={arXiv preprint arXiv:2107.10032},
  year={2021}
}

\bib{GR}{article}{
  title={Almost solutions of equations in permutations},
  author={L. Glebsky},
  author={L. M. Rivera},
  journal={Taiwanese Journal of Mathematics},
  volume={13},
  number={2A},
  pages={493--500},
  year={2009},
  publisher={The Mathematical Society of the Republic of China}
}

\bib{Gromov_hyp}{article}{
   author={Gromov, M.},
   title={Hyperbolic groups},
   conference={
      title={Essays in group theory},
   },
   book={
      series={Math. Sci. Res. Inst. Publ.},
      volume={8},
      publisher={Springer, New York},
   },
   date={1987},
   pages={75--263},
   review={\MR{919829}},
}

\bib{Gromov}{article}{
  author = {M. Gromov},
  year = {1999},
  pages = {109-197},
  title = {Endomorphisms of symbolic algebraic varieties},
  volume = {1},
  journal = {Journal of the European Mathematical Society},
  doi = {10.1007/PL00011162}
}

\bib{HS_C*}{article}{
   author={Hadwin, D.},
   author={Shulman, T.},
   title={Tracial stability for $C^*$-algebras},
   journal={Integral Equations Operator Theory},
   volume={90},
   date={2018},
   number={1},
   pages={Paper No. 1, 35},
   issn={0378-620X},
   review={\MR{3767651}},
   doi={10.1007/s00020-018-2430-1},
}
	
\bib{HS}{article}{
   author={D. Hadwin},
   author={T. Shulman},
   title={Stability of group relations under small Hilbert-Schmidt
   perturbations},
   journal={J. Funct. Anal.},
   volume={275},
   date={2018},
   number={4},
   pages={761--792},
   issn={0022-1236},
   review={\MR{3807776}},
   doi={10.1016/j.jfa.2018.05.006},
}

\bib{ISW}{article}{
  title={Cohomological obstructions to lifting properties for full $C^*$-algebras of property (T) groups},
  author={A. Ioana},
  author={P. Spaas},
  author={M. Wiersma},
  journal={Geometric and Functional Analysis},
  volume={30},
  number={5},
  pages={1402--1438},
  year={2020},
  publisher={Springer}
}

\bib{IS}{article}{
   author={A. Ioana},
   author={P. Spaas},
   title={${\rm II}_1$ factors with exotic central sequence algebras},
   journal={J. Inst. Math. Jussieu},
   volume={20},
   date={2021},
   number={5},
   pages={1671--1696},
   issn={1474-7480},
   review={\MR{4311565}},
   doi={10.1017/S1474748019000653},
}

\bib{Ioana1}{article}{
  title={Stability for product groups and property ($\tau$)},
  author={A. Ioana},
  journal={Journal of Functional Analysis},
  volume={279},
  number={9},
  pages={108729},
  year={2020},
  publisher={Elsevier}
}

\bib{Ioana2}{article}{
  title={Almost commuting matrices and stability for product groups},
  author={A. Ioana},
  journal={arXiv preprint arXiv:2108.09589},
  year={2021}
}

\bib{Jaikin}{article}{
   author={Jaikin-Zapirain, A.},
   title={The base change in the Atiyah and the L\"{u}ck approximation
   conjectures},
   journal={Geom. Funct. Anal.},
   volume={29},
   date={2019},
   number={2},
   pages={464--538},
   issn={1016-443X},
   review={\MR{3945838}},
   doi={10.1007/s00039-019-00487-3},
}

\bib{MIP*=RE}{article}{
  title={MIP*=RE},
  author={Ji, Z.},
  author={Natarajan, A.},
  author={Vidick, T.},
  author={Wright, J.},
  author={Yuen, H.},
  journal={Communications of the ACM},
  volume={64},
  number={11},
  pages={131--138},
  year={2021},
  publisher={ACM New York, NY, USA}
}

\bib{Kasparov}{article}{
   author={Kasparov, G. G.},
   title={Hilbert $C^{\ast} $-modules: theorems of Stinespring and
   Voiculescu},
   journal={J. Operator Theory},
   volume={4},
   date={1980},
   number={1},
   pages={133--150},
   issn={0379-4024},
   review={\MR{587371}},
}

\bib{Kechris}{article}{
   author={Kechris, A. S.},
   title={Topology and descriptive set theory},
   journal={Topology Appl.},
   volume={58},
   date={1994},
   number={3},
   pages={195--222},
   issn={0166-8641},
   review={\MR{1288299}},
   doi={10.1016/0166-8641(94)00146-4},
}

\bib{Kirchberg}{article}{
   author={Kirchberg, E.},
   title={Discrete groups with Kazhdan's property ${\rm T}$ and
   factorization property are residually finite},
   journal={Math. Ann.},
   volume={299},
   date={1994},
   number={3},
   pages={551--563},
   issn={0025-5831},
   review={\MR{1282231}},
   doi={10.1007/BF01459798},
}

\bib{KT}{article}{
   author={A. Klyachko},
   author={A. Thom},
   title={New topological methods to solve equations over groups},
   journal={Algebr. Geom. Topol.},
   volume={17},
   date={2017},
   number={1},
   pages={331--353},
   issn={1472-2747},
   review={\MR{3604379}},
   doi={10.2140/agt.2017.17.331},
}

\bib{KK}{article}{
   author={Kotowski, M.},
   author={Kotowski, M.},
   title={Random groups and property $(T)$: \.{Z}uk's theorem revisited},
   journal={J. Lond. Math. Soc. (2)},
   volume={88},
   date={2013},
   number={2},
   pages={396--416},
   issn={0024-6107},
   review={\MR{3106728}},
   doi={10.1112/jlms/jdt024},
}

\bib{KL}{article}{
   author={G. Kozma},
   author={A. Lubotzky},
   title={Linear representations of random groups},
   journal={Bull. Math. Sci.},
   volume={9},
   date={2019},
   number={3},
   pages={1950016, 12},
   issn={1664-3607},
   review={\MR{4045386}},
   doi={10.1142/S1664360719500164},
}
	
\bib{LV}{article}{
  title={Characters of solvable groups, Hilbert-Schmidt stability and dense periodic measures},
  author={A. Levit},
  author={I. Vigdorovich},
  journal={arXiv preprint arXiv:2206.02268},
  year={2022}
}

\bib{NPS}{article}{
  title={On $II_1$ factors arising from 2-cocycles of w-rigid groups},
  author={R. Nicoara},
  author={S. Popa},
  author={R. Sasyk},
  journal={Journal of Functional Analysis},
  volume={242},
  number={1},
  pages={230--246},
  year={2007},
  publisher={Elsevier}
}

\bib{Morris}{book}{
place={Cambridge},
series={London Mathematical Society Lecture Note Series},
title={Pontryagin Duality and the Structure of Locally Compact Abelian Groups},
DOI={10.1017/CBO9780511600722},
publisher={Cambridge University Press},
author={S. A. Morris},
year={1977},
}

\bib{Olesen}{article}{
  title={The Connes embedding problem, Sofic groups and the QWEP Conjecture},
  author={K. K. Olesen},
  year={Master thesis},
  eprint={http://web.math.ku.dk/~musat/thesis_final_KKO_March12.pdf}
}

\bib{Ollivier}{book}{
   author={Ollivier, Yann},
   title={A January 2005 invitation to random groups},
   series={Ensaios Matem\'{a}ticos [Mathematical Surveys]},
   volume={10},
   publisher={Sociedade Brasileira de Matem\'{a}tica, Rio de Janeiro},
   date={2005},
   pages={ii+100},
   isbn={85-85818-30-1},
   review={\MR{2205306}},
}


\bib{Ozawa}{article}{
  title={About the Connes embedding conjecture},
  author={N. Ozawa},
  journal={Japanese Journal of Mathematics},
  volume={8},
  number={1},
  pages={147--183},
  year={2013},
  publisher={Springer}
}

\bib{Paulsen}{book}{
   author={Paulsen, V.},
   title={Completely bounded maps and operator algebras},
   series={Cambridge Studies in Advanced Mathematics},
   volume={78},
   publisher={Cambridge University Press, Cambridge},
   date={2002},
   pages={xii+300},
   isbn={0-521-81669-6},
   review={\MR{1976867}},
}

\bib{Pestov}{article}{
  title={Hyperlinear and Sofic Groups: A Brief Guide},
  author={V. Pestov},
  journal={Bulletin of Symbolic Logic},
  year={2008},
  volume={14},
  pages={449 - 480}
}

\bib{PX}{article}{
 author = {Pisier, G.},
 author={Xu, Q.},
 title = {Non-commutative {{\(L^p\)}}-spaces},
 BookTitle = {Handbook of the geometry of Banach spaces. Volume 2},
 ISBN = {0-444-51305-1},
 pages = {1459--1517},
 year = {2003},
 publisher = {Amsterdam: North-Holland},
}

\bib{Popa_corr}{article}{
   author={Popa, S.},
   title={Correspondences},
   note={Preprint, INCREST 56 (1986), 1986.},
   eprint={https://www.math.ucla.edu/~popa/popa-correspondences.pdf}
}

\bib{Radulescu}{article}{
  title={The von Neumann algebra of the non-residually finite Baumslag group $\langle a,b\vert a b^3 a^{-1} = b^2 \rangle$ embeds into $\mathcal{R}^\omega$},
  author={Radulescu, F.},
  journal={arXiv preprint math/0004172},
  year={2000}
}

\bib{Rosendal}{article}{
   author={Rosendal, C.},
   title={Automatic continuity of group homomorphisms},
   journal={Bull. Symbolic Logic},
   volume={15},
   date={2009},
   number={2},
   pages={184--214},
   issn={1079-8986},
   review={\MR{2535429}},
   doi={10.2178/bsl/1243948486},
}

\bib{Shalom}{article}{
   author={Shalom, Y.},
   title={Rigidity of commensurators and irreducible lattices},
   journal={Invent. Math.},
   volume={141},
   date={2000},
   number={1},
   pages={1--54},
   issn={0020-9910},
   review={\MR{1767270}},
   doi={10.1007/s002220000064},
}

\bib{Takesaki}{book}{
   author={Takesaki, M.},
   title={Theory of operator algebras. I},
   series={Encyclopaedia of Mathematical Sciences},
   volume={124},
   note={Reprint of the first (1979) edition;
   Operator Algebras and Non-commutative Geometry, 5},
   publisher={Springer-Verlag, Berlin},
   date={2002},
   pages={xx+415},
   isbn={3-540-42248-X},
   review={\MR{1873025}},
}
	

\bib{Thom1}{article}{
  title={Examples of hyperlinear groups without factorization property},
  author={A. Thom},
  journal={Groups, Geometry, and Dynamics},
  volume={4},
  number={1},
  pages={195--208},
  year={2009}
}

\bib{Thom2}{article}{
   author={A. Thom},
   title={Finitary approximations of groups and their applications},
   conference={
      title={Proceedings of the International Congress of
      Mathematicians---Rio de Janeiro 2018. Vol. III. Invited lectures},
   },
   book={
      publisher={World Sci. Publ., Hackensack, NJ},
   },
   date={2018},
   pages={1779--1799},
   review={\MR{3966829}},
}

\bib{Weiss}{article}{
  author = {B. Weiss},
  year = {2000},
  pages = {350-359},
  title = {Sofic groups and dynamical systems},
  volume = {62},
  journal = {Sankhyā Ser. A},
  doi = {10.2307/25051326}
}

\bib{Zuk}{article}{
   author={\.{Z}uk, A.},
   title={Property (T) and Kazhdan constants for discrete groups},
   journal={Geom. Funct. Anal.},
   volume={13},
   date={2003},
   number={3},
   pages={643--670},
   issn={1016-443X},
   review={\MR{1995802}},
   doi={10.1007/s00039-003-0425-8},
}

\end{biblist}
\end{bibdiv}
\end{document}